\documentclass[smallextended]{svjour3}       
\usepackage[bottom=4cm, right=4cm, left=4cm, top=4cm]{geometry}
\usepackage{latexsym}
\usepackage{amssymb}
\usepackage{amsmath}
\usepackage[mathscr]{eucal}
\usepackage{graphicx}
\usepackage{hyperref}
\usepackage{caption}
\usepackage{subcaption}
\usepackage{setspace}

\renewcommand{\qed}{\hfill{\ \ \rule{2mm}{2mm}} \vspace{0.2in}}

\newcommand{\ind}{1\hspace{-2.3mm}{1}}

\renewcommand{\thefigure}{\arabic{figure}}
\begin{document}

\title{Maximum Spanning Trees of Random Geometric Graphs With Independent Edge Weights}
\titlerunning{Maximum Spanning Trees of Weighted Random Geometric Graphs}

\author{ \textbf{Ghurumuruhan Ganesan}}
\authorrunning{G. Ganesan}
\institute{IISER Bhopal,\\
\email{gganesan82@gmail.com }}

\date{}
\maketitle


\begin{abstract}
In this paper, we study maximum weight spanning trees of the random geometric graph (RGG)~\(G\) formed by~\(n\) vertices where each edge is independently either open or closed with a certain probability and is also equipped with an independent random positive weight. We use segmentation and iterative path construction to obtain deviation bounds for the order of growth of the maximum weight of a spanning tree  in terms of an inverse of the edge weight complementary cumulative distribution function (ccdf) and also illustrate our results for the special cases of power law and exponential decay. We then use martingale difference methods to individually estimate the variance contribution due to randomness in vertex locations and edge states/weights and determine sufficient conditions for~\(L^2-\)convergence of the maximum weight, appropriately scaled and centred.

\vspace{0.1in} \noindent \textbf{Key words:} Maximum weight spanning trees, Bernoulli random geometric graphs, deviation estimates, variance bounds, ~\(L^2-\)convergence.

\vspace{0.1in} \noindent \textbf{AMS 2000 Subject Classification:} Primary: 60D05, 60C05.
\end{abstract}

\bigskip

\renewcommand{\theequation}{\thesection.\arabic{equation}}
\setcounter{equation}{0}
\section{Introduction} \label{intro}
Spanning trees of graphs equipped with random weights are important from theoretical and application perspectives. Typically, the  weight of an edge is associated with the cost of performing a certain operation on the edge and in this context, minimum weight spanning trees have been extensively studied from the random graph perspective. In case of Bernoulli random graphs where the vertex locations are fixed but edge weights are random,~\cite{frieze} obtained constant limits for the weight of the minimum spanning tree when the cumulative distribution function (cdf) varied linearly near the origin. Later~\cite{aldous} use the local convergence method to obtain bounds for the generic edge weight case and since then many variants have been studied (see for example~\cite{berry} and references therein).

Detailed results regarding deviation, central limit theorems (CLTs) and scaling are also known for minimum spanning trees of Euclidean random graphs where the weight of an edge is proportional to a power~\(\alpha\) of its Euclidean length~\cite{steele,steele2,penrose2,yukich}. The main goal here is to establish how the asymptotics vary with the edge weight exponent~\(\alpha\) and it is known that for all~\(\alpha < 2\) the minimum spanning tree is heavy in the sense that its weight \emph{grows} with the number of vertices~\(n.\) Regarding related results, the paper~\cite{kes_lee} uses martingales to obtain a CLT result for the minimum weight spanning tree and recently in our paper~\cite{ganesan}, we obtain bounds for the minimum weight of spanning trees of random geometric graphs (RGGs) in terms of the adjacency distance.


In this paper, we are interested in studying \emph{maximum weight} spanning trees of RGGs where each edge is independently present with a certain probability and is also equipped with an independent random edge weight. This has applications in wireless networks undergoing independent \emph{shadowing} and \emph{fading} where edge state (present or absent) denotes whether the communication link is operable or not (due to shadowing) and the edge weight denotes the fading gain (we refer to Chapter~\(7,\)~\cite{goldsmith} for more information on related aspects). The maximum weight spanning tree  represents the maximum throughout in a minimally connected network and it is of interest to study how the edge state and weights affect the overall weight of the spanning tree.

We use segmentation and iterative methods to establish deviation estimates for the order of growth of the maximum weight spanning tree.  We then use martingale difference techniques to obtain variance bounds and thereby determine sufficient conditions for~\(L^2-\)convergence of the maximum weight, appropriately scaled and centred. We also illustrate our results using power law and exponential decay as examples.

In the following section, we state and prove our main results regarding deviation and~\(L^2-\)convergence of the maximum weight of spanning trees in random geometric graphs equipped with independent edge weights.

\renewcommand{\theequation}{\thesection.\arabic{equation}}
\setcounter{equation}{0}
\section{Weighted Random Geometric Graphs}\label{sec_mst_max}
Let~\(K_n\) be the complete graph with vertex set~\(\{1,2,\ldots,n\}\) and let~\(\{Z(f)\}_{f \in K_n}\) be independent and identically distributed (i.i.d.) Bernoulli random variables indexed by the edge set of~\(K_n\) and having distribution
\begin{equation}\label{x_dist}
\mathbb{P}(Z(f) = 1) = p_n= 1-\mathbb{P}(Z(f)=0)
\end{equation}
for some~\(0 < p_n  < 1.\) If edge~\(f=(u,v)\) has endvertices~\(u\) and~\(v,\) then we also denote~\(Z(f) = Z(u,v)\) and equip~\(f\) with a positive random weight~\(W(f) = W(u,v)\) that is independent of the edge states~\(\{Z(e)\}_{e \in K_n}.\) The random variables~\(\{W(f)\}_{f \in K_n}\)  are  i.i.d.\ with a common complementary cumulative distribution function (ccdf)~\(F_c\) defined by \[F_c(x) := \mathbb{P}(W(f) > x)\] for~\(x > 0.\)

Let~\(\{X_i\}_{1 \leq i \leq n}\) be i.i.d.\ with a common density~\(f\) in the unit square~\(S = [0,1]^2\) satisfying
\begin{equation}\label{f_eq}
\epsilon_1 \leq f(x) \leq \epsilon_2
\end{equation}
for all~\(x \in S\) and some positive finite constants~\(\epsilon_1,\epsilon_2.\) We define~\(X_u\) to be the random \emph{location} of the vertex~\(u\) and let the Euclidean distance~\(d(X_u,X_v)\) between~\(X_u\) and~\(X_v\) denote the \emph{length} of the edge~\((u,v).\) Let~\(0 < r_n < 1\) be any positive sequence and let~\(G = G(r_n)\) be the random subgraph of~\(K_n\) formed by the set of all edges~\(f = (u,v)\) satisfying
\[d(X_u,X_v) < r_n \text{ and } Z(u,v) = 1.\] We define~\(G\) to be the Bernoulli Random Geometric Graph (BRGG) with adjacency distance~\(r_n\) and edge probability~\(p_n.\)

A connected acyclic subgraph of~\(G\) is a tree. We define the weight~\(W({\cal T})\) of a tree~\({\cal T}\) to be the sum of the weights of edges in~\({\cal T};\) i.e.,~\(W({\cal T}) := \sum_{f \in {\cal T}} W(f).\) We say that a tree~\({\cal T}\) is a \emph{spanning tree} of a component~\({\cal C} \subset G,\) if~\({\cal T}\) contains all the vertices of~\({\cal C}.\)

Letting~\(\tau_n\) denote the maximum weight of a spanning tree of the largest component of~\(G,\) we have the following result. Throughout constants do not depend on~\(n\) and for~\(z > 0\)  we define~
\begin{equation}\label{h_def}
H(z) := \max\left\{x : F_c(x) \geq \frac{1}{z}\right\}.
\end{equation}
\begin{theorem}\label{thm_mst_max} Let~\(\epsilon_1,\epsilon_2\) be as in~(\ref{f_eq}) and suppose~\(nr_n^2p_n \geq M\log{n}\) for
some constant~\(M  >0.\) Also let~\(E_{con}\) denote the event that~\(G\) is connected.\\
\((a)\)  There  are constants~\(\gamma_0,\gamma_1 > 0\) such that if~\(M > \gamma_0,\) then
\begin{equation}\label{low_dev}
\mathbb{P}\left(E_{con} \bigcap \left\{\tau_n \geq \gamma_1 n H\left(\gamma_1nr_n^2p_n\right)\right\} \right) \geq 1-e^{-\gamma_1 nr_n^2p_n}
\end{equation}
and~\(\mathbb{E}\tau_n \geq \gamma_1 n H\left(\gamma_1nr_n^2p_n\right).\)\\
\((b)\) Suppose there are constants~\(C, x_0 > 0\) and~\(s > 2\) strictly, such that
\begin{equation}\label{dilpax}
F_c(ax) \leq \frac{C}{a^{s}}F_c(x)
\end{equation}
all~\(a > 1\) and~\(x > x_0.\) There are constants~\(\gamma_2,\gamma_3 > 0\) such that if~\(M > \gamma_2\) then
\begin{equation}\label{up_dev}
\mathbb{P}\left(E_{con} \bigcap \left\{\tau_n \leq \gamma_2 n H\left(nr_n^2p_n\right)\right\} \right) \geq 1-\frac{1}{(n^2r_n^2p_n)^{\gamma_3}}
\end{equation}
and~\(\mathbb{E}\tau_n \leq \gamma_2nH(nr_n^2p_n).\)
\end{theorem}
The~\(H(.)\) factors in~(\ref{low_dev}) and~(\ref{up_dev}) could be interpreted as the ``gain" obtained due to the heaviness of the edge weights as opposed to constant edge weights. 





We illustrate Theorem~\ref{thm_mst_max} with two examples below involving power law and exponential decay, respectively.

\underline{\emph{Power Law}}: Suppose~\( \frac{A_1}{x^{s}} \leq  F_c(x) \leq  \frac{A_2}{x^{s}}\) for some constants~\(A_1,A_2 > 0\) and all~\(x\) large. In this case~\(B_1 z^{1/s} \leq H(z) \leq B_2 z^{1/s}\) for some constants~\(B_1,B_2 >0\) and all~\(z\) large. Moreover,~\[F_c(ax) \leq \frac{A_2}{(ax)^{s}}  = \frac{A_2/A_1}{a^{s}} \cdot \frac{A_1}{x^{s}} \leq \frac{A_2/A_1}{a^{s}} \cdot F_c(x)\] for all~\(x\) large and any~\(a > 1.\) Therefore if~\(nr_n^2p_n \geq M\log{n}\) for some large enough constant~\(M,\) then from the deviation bounds in~(\ref{low_dev}) and~(\ref{up_dev}), we get that there are constants~\(\delta_1,\delta_2 > 0\) such that~\(G\) is connected and
\begin{equation}\label{tau_n_power_law}
\delta_1 \left(nr_n^2p_n\right)^{1/s} \leq \frac{\tau_n}{n} \leq \delta_2 (nr_n^2p_n)^{1/s}
\end{equation}
with high probability, i.e., with probability converging to one as~\(n \rightarrow \infty.\)

\underline{\emph{Exponential}}: Suppose~\( e^{-A_1x^{\alpha}} \leq  F_c(x) \leq  e^{-A_2x^{\alpha}}\) for some constants~\(\alpha,A_1,A_2 > 0\) and all~\(x\) large. In this case~\(B_1 (\log{z})^{1/\alpha} \leq H(z) \leq B_2 (\log{z})^{1/\alpha}\) for some constants~\(B_1,B_2 >0\) and all~\(z\) large. For any~\(s > 0, a > 1\) and all~\(x\) large, we have that~\[F_c(ax) \leq e^{-A_2a^{\alpha}x^{\alpha}} \leq \frac{1}{a^{s}} \cdot e^{-A_1x^{\alpha}} \leq \frac{1}{a^{s}} F_c(x).\] Therefore if~\(nr_n^2p_n \geq M\log{n}\) for some large enough constant~\(M,\) then from the deviation bounds in~(\ref{low_dev}) and~(\ref{up_dev}), we get that there are constants~\(\delta_1,\delta_2 > 0\) such that~\(G\) is connected and
\begin{equation}\label{tau_n_exponential}
\delta_1 \left(\log(nr_n^2p_n)\right)^{1/\alpha} \leq \frac{\tau_n}{n} \leq \delta_2 \left(\log(nr_n^2p_n)\right)^{1/\alpha}
\end{equation}
with high probability.









Throughout, we use the following deviation estimate regarding sums of independent Bernoulli random variables. Let~\(\{S_j\}_{1 \leq j \leq r}\) be independent Bernoulli random variables satisfying~\(\mathbb{P}(S_j = 1) = 1-\mathbb{P}(S_j = 0) > 0.\) If~\(T_r := \sum_{j=1}^{r} S_j, \theta_r := \mathbb{E}T_r\) and~\(0 < \gamma \leq \frac{1}{2},\) then
\begin{equation}\label{conc_est_f}
\mathbb{P}\left(\left|T_r - \theta_r\right| \geq \theta_r \gamma \right) \leq 2\exp\left(-\frac{\gamma^2}{4}\theta_r\right)
\end{equation}
for all \(r \geq 1.\) For a proof of~(\ref{conc_est_f}), we refer to Corollary A.1.14, pp. 312 of Alon and Spencer (2008).

\emph{Proof of Theorem~\ref{thm_mst_max}\((a)\)}: We begin by tiling~\(S=[0,1]^2\) into small disjoint~\(\frac{r_n}{4} \times \frac{r_n}{4}\) squares~\(\{R_i\}_{1 \leq i \leq N}\) where we assume for simplicity that~\(N = \frac{16}{r_n^2}\) is an integer; else we choose the side length of~\(R_i\) from the interval~\(\left(\frac{r_n}{5}, \frac{r_n}{4}\right]\) appropriately so that~\(N\) is an integer. This is possible since~\(\left(\frac{5}{r_n}\right)^2 - \left(\frac{4}{r_n}\right)^2 = \frac{9}{r_n^2} \geq 1.\) We label the squares as in Figure~\ref{fig_squares} so that~\(R_i\) and~\(R_{i+1}\) share an edge for each~\(1 \leq i \leq N-1.\) We also say that a vertex~\(u\) belongs to the square~\(R_i\) if the corresponding location~\(X_u \in R_i.\)

\begin{figure}[tbp]
\centering
\includegraphics[width=3in, trim= 20 200 50 110, clip=true]{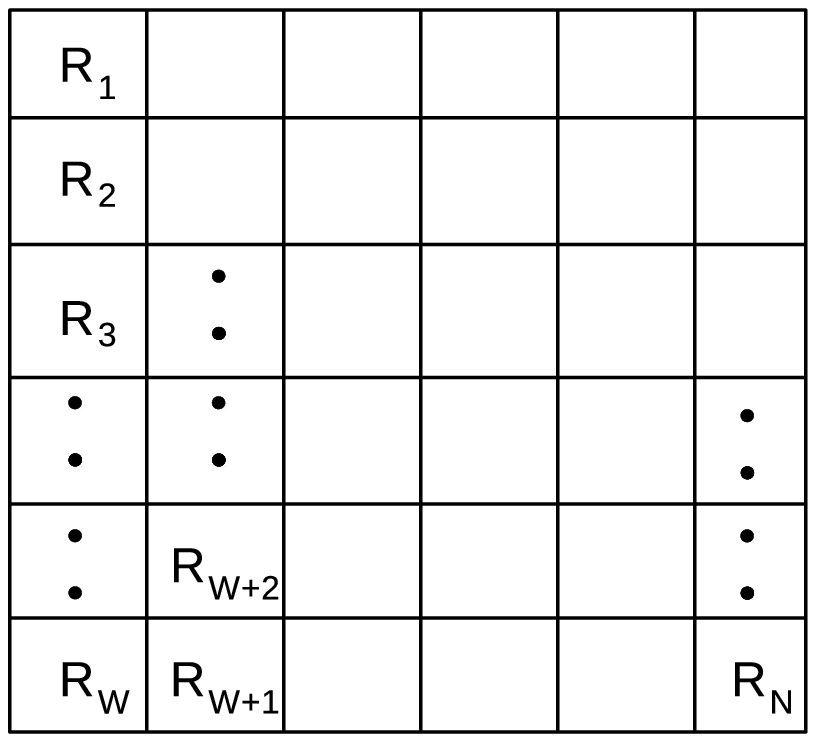}
\caption{Tiling the unit square into~\(N = \frac{n}{A^2}\) smaller~\(\frac{A}{\sqrt{n}} \times \frac{A}{\sqrt{n}}\) squares~\(\{R_l\}_{1 \leq l \leq \frac{n}{A^2}}.\)}
\label{fig_squares}
\end{figure}

The main idea in our proof is that we estimate the maximum weight of a spanning tree in each small square~\(R_i\) and then combine all these subtrees to get an overall spanning tree, thereby obtaining  a lower bound for~\(\tau_n.\)

For future use in estimating the variance  in Theorem~\ref{thm_var_est} later, we establish here connectivity estimates for subgraphs of~\(G\) obtained after removal of  a set of vertices of constant size. Let~\({\cal V} \subset \{1,2,\ldots,n\}\) be set of constant size~\(k = \#{\cal V}\) and let~\(G({\cal V})\) be the subgraph of~\(G\) obtained after removing the vertices from~\({\cal V}.\) Any vertex~\(u \notin {\cal V}\) is present in~\(R_i\) with probability~\(\int_{R_i}f\) and so the number of vertices~\(N(R_i) = N(R_i,{\cal V})\) in~\(R_i\) is Binomially distributed with parameters~\(n-k\) and~\( \int_{R_i}f \in [\epsilon_1 \frac{r_n^2}{16},\epsilon_2 \frac{r_n^2}{16}]\) by the density bounds in~(\ref{f_eq}).

Defining
\begin{equation}\label{e_vert_i_def}
E_{vert}(i,{\cal V}) := \left\{\frac{\epsilon_1 nr_n^2}{32} \leq N(R_i) \leq 2\epsilon_2 nr_n^2  \right\},
\end{equation}
we get from the deviation estimate~(\ref{conc_est_f}) that
\begin{equation}\label{e_vert_i_est}
\mathbb{P}\left(E_{vert}(i,{\cal V})\right) \geq 1- \exp\left(-2C_1 nr_n^2\right)
\end{equation}
for some constant~\(C_1 > 0.\) Here and henceforth constants do not depend on the choice of~\(i\) or~\({\cal V}.\) Further defining~\(E_{vert}  = E_{vert}({\cal V}) := \bigcap_{1 \leq i \leq N} E_{vert}(i,{\cal V}),\) we get from the union bound that
\[\mathbb{P}(E_{vert}) \geq 1-N e^{-2C_1nr_n^2}.\] From Theorem statement~\(nr_n^2 \geq nr_n^2p_n \geq M\log{n}\) and so~\(N = \frac{16}{r_n^2} \leq n\) for all~\(n\) large. Thus
\begin{equation}\label{e_vert_est}
\mathbb{P}(E_{vert}) \geq 1- n \cdot e^{-2C_1nr_n^2} \geq 1- e^{-C_1 nr_n^2}
\end{equation}
for all~\(n\) large.

Let~\(\mathbb{P}_{vert}(.) := \mathbb{P}(.\mid E_{vert})\) be the probability measure conditioned on the occurrence of the event~\(E_{vert}\) and for~\(1 \leq i \leq N,\) let~\(G_i = G_i({\cal V})\) be the induced subgraph of~\(G({\cal V})\) formed by the vertices in~\(R_i.\) In what follows, we estimate the probability of the event~\(E_{con}(i,{\cal V})\) that~\(G_i\) is connected. Let~\(n_i := N(R_i) = N(R_i,{\cal V})\) be the number of vertices of~\({\cal V}^c \subseteq \{1,2,\ldots,n\}\) present in the square~\(R_i.\) By choice, the distance between any two vertices in~\(G_i\) is at most~\(r_n\) and so only the edge states determine the connectivity of~\(G_i.\) Arguing as in Chapter~\(7,\) pp.~\(165,\)~\cite{boll}, we see that if~\(G_i\) is disconnected, then~\(G_i\) contains  a tree~\({\cal L}\) of size~\(3 \leq l \leq \frac{n_i}{2}\) such that no vertex of~\({\cal L}\) is adjacent to any vertex not in~\({\cal L}\) and since there are~\(l^{l-2}\) possible labelled trees containing~\(l\) vertices, we get
\begin{equation}\label{gilma}
\mathbb{P}(E_{con}^c(i,{\cal V}) \mid n_i) \leq \sum_{l=3}^{n_i/2}{ n_i \choose l} l^{l-2} p_n^{l-1} (1-p_n)^{l(n_i-l)}.
\end{equation}
From standard Binomial estimates we have~\({n_i\choose l} \leq \left(\frac{n_i e}{l}\right)^{l}\)  and since~\(l \leq \frac{n_i}{2},\) we also have that  \[(1-p_n)^{l(n_i-l)} \leq e^{-p_nl(n_i-l)} \leq \exp\left(-\frac{p_nln_i}{2}\right).\] Thus
\begin{align}\label{gilpa}
\mathbb{P}(E_{con}^c(i,{\cal V}) \mid n_i) &\leq \sum_{l=3}^{n_i/2}\frac{1}{p_n l^2} \left(en_ip_ne^{-n_ip_n/2}\right)^{l} \nonumber\\
&\leq \frac{1}{p_n} \sum_{l=3}^{n_i/2} \left(en_ip_ne^{-n_ip_n/2}\right)^{l}.
\end{align}

If~\(E_{vert}\) occurs, then~\(n_i  = N(R_i) \geq \frac{\epsilon_1 nr_n^2}{2}\) and so from Theorem statement we see that~\[n_ip_n \geq \frac{\epsilon_1 nr_n^2p_n}{2} \geq \frac{M\epsilon_1 \log{n}}{2}.\] This implies that for all~\(n\) large, \[en_ip_ne^{-n_ip_n/2} \leq e^{-n_ip_n/4} \leq e^{-2C_2nr_n^2p_n}\] for some constant~\(C_2 >0.\) Again using~\(nr_n^2p_n \geq M\log{n}\) we also see that \[\frac{1}{p_n} \leq \frac{nr_n^2}{\log{n}} \leq n.\] Plugging these bounds into~(\ref{gilpa}), we get that
\begin{align}
\mathbb{P}_{vert}(E^c_{con}(i,{\cal V})) &\leq n \sum_{l\geq 3} (e^{-2C_2nr_n^2p_n})^{l} \nonumber\\
&= \frac{n e^{-6C_2nr_n^2p_n}}{1-e^{-2C_2nr_n^2p_n}} \nonumber\\
&\leq 2ne^{-6C_2nr_n^2p_n}
\end{align}
for all~\(n\) large, since~\(nr_n^2p_n \geq M\log{n}\) by Theorem statement. Defining
\[E_{con} = E_{con}({\cal V}) := \bigcap_{1 \leq i \leq N} E_{con}(i,{\cal V}),\]
we then get from the union the bound and the estimate~\(N \leq n\) (see discussion prior to~(\ref{e_vert_est})) that
\begin{align}
\mathbb{P}_{vert}(E_{con}) &\geq 1- 2Nne^{-2C_2nr_n^2p_n} \nonumber\\
&\geq 1- 2n^2e^{-2C_2nr_n^2p_n} \nonumber\\
&\geq 1- e^{-C_2nr_n^2p_n} \label{e_con_est}
\end{align}
since~\(nr_n^2p_n \geq M\log{n}\) by Theorem statement.

Next we look at ``cross edges"  that connect vertices in~\(R_i\) with vertices in the adjacent square~\(R_{i+1}.\) Formally, for~\(1 \leq i \leq N-1\) we say that an edge~\(f = (u,v) \in G({\cal V})\) is an~\(i^{th}-\)cross edge if~\(X_u \in R_i\) and~\(X_v \in R_{i+1}\) or~\(X_v \in R_i\) and~\(X_u \in R_{i+1}.\) Given that~\(E_{vert}\) occurs, there are~\(n_i \geq \frac{\epsilon_1 nr_n^2}{2}\) vertices of~\({\cal V}^c\) each in~\(R_i\) and~\(R_{i+1}\) and so if~\(E_{cross}(i,{\cal V})\) denotes the event that~\(G({\cal V})\) contains an~\(i^{th}-\)cross edge, then
\begin{align}
\mathbb{P}_{vert}(E^c_{cross}(i,{\cal V})) &\leq (1-p_n)^{\left(\frac{\epsilon_1 nr_n^2}{2}\right)^2} \nonumber\\
&\leq \exp\left(-\left(\frac{\epsilon_1 nr_n^2}{2}\right)^2p_n\right) \nonumber\\
&\leq e^{-2nr_n^2p_n},
\end{align}
since~\(nr_n^2 \geq nr_n^2p_n \geq M\log{n}\) by Theorem statement. Defining~\[E_{cross} = E_{cross}({\cal V}) := \bigcap_{1 \leq i \leq N-1}E_{cross}(i,{\cal V})\] and arguing as in~(\ref{e_con_est}) we get that
\begin{equation}\label{e_cross_est}
\mathbb{P}_{vert}(E_{cross}) \geq 1- Ne^{-2nr_n^2p_n} \geq 1-e^{-nr_n^2p_n}
\end{equation}
for all~\(n\) large.

Combining~(\ref{e_cross_est}) with~(\ref{e_con_est}) and using the union bound we see that
\[\mathbb{P}_{vert}(E_{con} \cap E_{cross}) \geq 1-e^{-C_2nr_n^2p_n} - e^{-nr_n^2p_n}.\] Finally defining~\(E_{tot}  = E_{tot}({\cal V}) := E_{vert} \bigcap E_{con} \bigcap E_{cross}\) we get from the estimate~(\ref{e_vert_est}) for~\(E_{vert}\) and the union bound that
\begin{align}\label{e_tot_est}
\mathbb{P}(E_{tot}) &\geq 1-e^{-C_1nr_n^2} - e^{-C_2nr_n^2p_n} - e^{-nr_n^2p_n} \nonumber\\
&\geq 1-e^{-C_3nr_n^2p_n}
\end{align}
for all~\(n\) large and some constant~\(C_3 > 0.\)

We assume henceforth that~\({\cal V} = \emptyset\) and that~\(E_{tot}\) occurs so that for each~\(1 \leq i \leq N,\) the subgraph~\(G_i \subset G\) is connected.  Let~\({\cal T}_i\) be the maximum weight spanning tree of~\(G_i\) and for~\(1 \leq i \leq N-1\) let~\(f_i\) be any~\(i^{th}-\)cross edge whose existence is guaranteed by the occurrence of~\(E_{cross} \supset E_{tot}.\) The union \[{\cal T}_{tot} := \bigcup_{i=1}^{N} {\cal T}_i \bigcup \bigcup_{i=1}^{N-1} \{f_i\}\] is a spanning tree of the overall graph~\(G\) and so if~\(\tau_n(i)\) is the maximum weight of a spanning tree of~\(G_i,\) we get that
\begin{equation}\label{neha_tits}
\tau_n \geq \sum_{i=1}^{N} \tau_n^{(i)}.
\end{equation}

For each~\(i,\) we now estimate~\(\tau_n^{(i)}\) ``indirectly" as follows. We recall that~\(n_i = N(R_i)\) is the number of vertices present in~\(R_i\) and so we let~\[{\cal U}_i := \{u(1),\ldots,u(n_i)\} \subset \{1,2,\ldots,n\}\] be the random vertices present in~\(R_i.\) As before~\(G_i\) is the induced subgraph of~\(G\) with vertex set~\({\cal U}_i\) and for a vertex~\(v \in {\cal U}_i\) let~\({\cal N}_i(v)\) be the set of neighbours of~\(v\) in~\(G_i.\) By our choice of the side length of~\(\frac{r_n}{4}\) for the square~\(R_i,\) any two vertices in~\({\cal U}_i\) are within a distance of~\(r_n\) from each other and so for each~\(v,\) the neighbour set~\({\cal N}_i(v)\) depends solely on the state of edges having~\(v\) as one endvertex and having the other endvertex in~\({\cal U}_i \setminus \{v\}.\)

Our strategy of estimating~\(\tau_n^{(i)}\) is as follows. We first construct a path~\({\cal P}_i  \subset G_i\) containing enough heavy edges and obtain a lower bound for the total weight of~\({\cal P}(i).\) We then argue that if~\(G_i\) is connected, then~\({\cal P}_i\) is a subset of a spanning tree of~\(G_i\) and thereby obtain an indirect lower bound for~\(\tau_n^{(i)}.\) Details follow. For notational simplicity, we suppress the dependence on~\(i\)  on all terms that follow except for those already mentioned before.

Let~\(v_{1} := u(1)\) and for~\(j \geq 1\) let~\(v_{j+1}\) be the neighbour of~\(v_{j}\) in~\(G_i\) satisfying
\begin{equation}\label{wj_def}
W_j := W(v_{j},v_{j+1}) = \max_{v \in {\cal N}_i(v_j) \setminus \{v_1,\ldots,v_{j-1}\}} W(v_j,v)
\end{equation}
with the notation that~\(\{v_1,\ldots,v_{j-1}\} = \emptyset\) for~\(j=1.\) In words, among all neighbours of~\(v_j\) not encountered so far, we choose that edge with the largest weight. Let~\(L \leq n_i \leq n\) be the smallest integer such that~\({\cal N}(v_L) \subset \{v_1,\ldots,v_{L-1}\}\) and define the path~\({\cal P}_i  := (v_1,\ldots,v_L).\) For completeness, we also set~\(W_{j} :=0\) and~\(v_j := v_L\) for~\(L \leq j \leq n-1.\)

We estimate the weight of~\({\cal P}_i\) (i.e., the sum of weights of edges in~\({\cal P}_i\)) as follows. Let~\(N_j(v_j)\) be the number of neighbours of~\(v_j\) in~\({\cal U}_i \setminus \{v_1,\ldots,v_{j-1}\}.\) As mentioned before,~\(N_j(v_j)\) depends only the state of edges and not on the individual vertex locations and so given the ``history"~\({\cal U}_i \cup \{v_1,\ldots,v_{j}\},\) the term~\(N_j(v_j)\) is binomially distributed with parameters~\(n_i-j\) and~\(p_n.\) Therefore defining~\({\cal G}_j\) to be the sigma-field generated by~\({\cal U}_i \cup \{v_1,\ldots,v_{j}\},\) we see that~\(\mathbb{E}(N_j(v_j) \mid {\cal G}_j) = (n_i-j)p_n\) and so denoting~\[E_j(v_j) := \left\{N_j(v_j) \geq \frac{(n_i-j)p_n}{2}\right\},\] we get from the deviation estimate~(\ref{conc_est_f}) that
\begin{equation}\label{ej_vj_est}
\mathbb{P}\left(E_j(v_j) \mid {\cal G}_j\right) \geq 1-\exp\left(-\frac{(n_i-j)p_n}{16}\right).
\end{equation}

If~\(E_j(v_j)\) occurs, then the weight~\(W_j\) is the maximum of at least~\(\frac{(n_i-j)p_n}{2}\) i.i.d.\ random variables, each with ccdf~\(F_c\) and so
\begin{align}
\mathbb{P}\left(\{W_j \leq y\} \bigcap E_j(v_j) \mid {\cal G}_j\right) &\leq (1-F_c(y))^{(n_i-j)p_n/2} \nonumber\\
&\leq \exp\left(-F_c(y)\frac{(n_i-j)p_n}{2}\right). \label{wj_y_est}
\end{align}
Recalling the event~\(E_{vert}(i)\) defined in~(\ref{e_vert_i_def}), we now suppose that~\(E_{vert}(i)\) also occurs so that~\(n_i =N(R_i) \geq \frac{\epsilon_1 nr_n^2}{2}.\) By definition~\(E_{vert}(i)\) depends only the vertices in~\({\cal U}_i\) and therefore belongs to~\({\cal G}_j.\) Moreover, by Theorem statement we also get that~\[p_n \geq \frac{M\log{n}}{nr_n^2} \geq \frac{M\epsilon_1 \log{n}}{2n_i}.\]

Based on the definition of~\(H(.)\) in~(\ref{h_def}) we now set~\[y := \frac{1}{2}H\left(\frac{(n_i-j)p_n}{4}\right)\] and get that
\begin{equation}\label{daldaps2}
1 \geq F_c(y) \geq \frac{4}{(n_i-j)p_n}
\end{equation}
for all~\(1 \leq j \leq n_i-\frac{1}{p_n},\) since~\(n_ip_n\) is at least of the order of than~\(\log{n}.\) For~\(j \leq \frac{3n_i}{4}\) we also have that~\((n_i-j)p_n \geq \frac{n_ip_n}{4}\) and so using the property that~\(H(z)\) is non-decreasing in~\(z,\) we get
\begin{equation}\label{daldaps}
H\left(\frac{(n_i-j)p_n}{4}\right) \geq H\left(\frac{n_ip_n}{16}\right) \geq H\left(\frac{3\epsilon_1nr_n^2p_n}{32}\right) =: 2z_n,
\end{equation}
since~\(E_{vert}(i)\) occurs and so~\(n_i \geq \frac{\epsilon_1nr_n^2}{2}.\)

Plugging~(\ref{daldaps}) and~(\ref{daldaps2}) into~(\ref{wj_y_est}) we then get
\begin{equation}\label{daldaps3}
\mathbb{P}\left(\left\{W_j \leq z_n\right\} \bigcap E_j(v_j) \mid {\cal G}_j\right) \ind(E_{vert}(i)) \leq e^{-2} \ind(E_{vert}(i))
\end{equation}
for all~\(1 \leq j \leq \frac{3n_i}{4},\) where~\(\frac{3n_i}{4} \geq \frac{3\epsilon_1nr_n^2}{8} =:L_0,\) since~\(E_{vert}(i)\) occurs. This in turn implies that
\begin{align}\label{daldaps4}
&\mathbb{P}\left(W_j \leq z_n \mid {\cal G}_j\right) \ind(E_{vert}(i)) \nonumber\\
&\;\;\;\;\;=\;\;\mathbb{P}\left(\left\{W_j \leq z_n\right\} \bigcap E_j(v_j) \mid {\cal G}_j\right) \ind(E_{vert}(i)) \nonumber\\
&\;\;\;\;\;\;\;\;\;\;\;\;\;\;\;\;\;\;\;\;\; + \;\;\;\mathbb{P}\left(\left\{W_j \leq z_n\right\} \bigcap E^c_j(v_j) \mid {\cal G}_j\right) \ind(E_{vert}(i)) \nonumber\\
&\;\;\;\leq\;\;e^{-2}\ind(E_{vert}(i))  + \mathbb{P}\left(E^c_j(v_j) \mid {\cal G}_j\right) \ind(E_{vert}(i))
\end{align}
for all~\(1 \leq j \leq L_0.\)

From the estimate~(\ref{ej_vj_est}) for~\(E_j(v_j),\) we see that
\begin{align}
\mathbb{P}\left(E^c_j(v_j) \mid {\cal G}_j\right) \ind(E_{vert}(i)) &\leq \exp\left(-\frac{(n_i-j)p_n}{16}\right) \ind(E_{vert}(i)) \nonumber\\
&\leq \exp\left(-\frac{n_ip_n}{64}\right)\ind(E_{vert}(i)) \nonumber\\
&\leq \exp\left(-\frac{\epsilon_1 nr_n^2p_n}{128}\right)\ind(E_{vert}(i)) \label{neha_tits_two}
\end{align}
for all~\(1 \leq j \leq L_0 \leq \frac{3n_i}{4},\) where the second expression in~(\ref{neha_tits_two}) is true since~\(n_i-j \geq \frac{n_i}{4}\) and the final relation in~(\ref{neha_tits_two}) is true since~\(E_{vert}(i)\) occurs and so~\(n_i \geq \frac{\epsilon_1nr_n^2}{2}.\) From Theorem statement we know that~\(nr_n^2p_n \geq M\log{n}\) and so choosing~\(M\) large enough, we get from~(\ref{neha_tits_two}) and~(\ref{daldaps4}) that
\begin{equation}\label{jyothi_tits}
\mathbb{P}\left(W_j \leq z_n \mid {\cal G}_j \right) \ind(E_{vert}(i)) \leq 2e^{-2} \ind(E_{vert}(i))
\end{equation}
for all~\(n\) large and all~\(1 \leq j \leq L_0.\)

For~\(1 \leq j \leq L_0,\) we  define~\(Y(j) := \ind(W_j \geq z_n)\) so that~\[Q(j) := \sum_{l=1}^{j}Y(l)\] is the number of edges amongst the first~\(j\) edges of the path~\({\cal P}_i,\) each of whose weight is at least~\(z_n.\)  Letting~\(\theta,\lambda> 0\) be constants to be determined later, we use the Chernoff bound to obtain for~\(1 \leq j \leq L_0\) that
\begin{equation}\label{blue_zero}
\mathbb{P}\left(Q(j) \leq \theta j\right)  \leq e^{\lambda\theta j} \mathbb{E}\exp\left(-\lambda Q(j)\right).
\end{equation}
We have that
\begin{align}
\mathbb{E}\exp\left(-\lambda Q(j)\right) &= \mathbb{E}e^{-\lambda Q(j)} \ind(E_{vert}(i)) + \mathbb{E}e^{-\lambda Q(j)}\ind(E^c_{vert}(i)) \nonumber\\
&\leq \mathbb{E}e^{-\lambda Q(j)} \ind(E_{vert}(i)) + \mathbb{P}(E_{vert}^c(i)) \nonumber\\
&\leq \mathbb{E}e^{-\lambda Q(j)} \ind(E_{vert}(i)) + e^{-4Cnr_n^2}\label{blue_one}
\end{align}
for some constant~\(C > 0\) not depending on the choice of~\(i\) or~\(j,\) by~(\ref{e_vert_i_est}).

If~\({\cal H}_j\) is the sigma-field generated by the union of the sigma-field~\({\cal G}_j\) and the weights of the edges of~\({\cal P}_i\) having at least one endvertex in~\(\{v_1,\ldots,v_{j-1}\},\) then~\(Y(l)\) is~\({\cal H}_j-\)measurable for each~\(1 \leq l \leq j-1\) and so
\begin{equation}\label{red_one}
\mathbb{E}e^{-\lambda Q(j)} \ind(E_{vert}(i)) = \mathbb{E}e^{-\lambda Q(j-1)} \ind(E_{vert}(i))\mathbb{E}\left(e^{-\lambda Y(j)} \mid {\cal H}_j\right)
\end{equation}
since~\(E_{vert}(i) \in {\cal G}_j \subset {\cal H}_j\) as well. Taking expectations with respect to~\({\cal G}_j\) in~(\ref{jyothi_tits}) we see that~(\ref{jyothi_tits}) holds with~\({\cal G}_j\) replaced by~\({\cal H}_j\) and so
\begin{align}
\ind(E_{vert}(i))\mathbb{E}\left(e^{-\lambda Y(j)} \mid {\cal H}_j\right) &= \ind(E_{vert}(i))e^{-\lambda}\mathbb{P}\left(W_j \geq z_n \mid {\cal H}_j\right)  \nonumber\\
&\;\;\;\;\;\;\;\;\;\;\;\;\;\;\;+\;\;\;\ind(E_{vert}(i)) \mathbb{P}\left(W_j \leq z_n \mid {\cal H}_j\right) \nonumber\\
&\leq \ind(E_{vert}(i))(e^{-\lambda}  + 2e^{-2}) \nonumber\\
&\leq 3e^{-2}\ind(E_{vert}(i))\label{blue_two}
\end{align}
provided we set~\(\lambda= 3.\)

Plugging~(\ref{blue_two}) into~(\ref{red_one}), we get that
\[\mathbb{E}e^{-\lambda Q(j)}\ind(E_{vert}(i)) \leq 3e^{-2}\mathbb{E}e^{-\lambda Q(j-1)}\ind(E_{vert}(i))\]
for all~\(1 \leq j \leq L_0\) and proceeding iteratively, we get that
\[\mathbb{E}e^{-\lambda Q(j)} \ind(E_{vert}(i)) \leq (3e^{-2})^{j}.\] Substituting this into~(\ref{blue_one}) and using~(\ref{blue_zero}), we get that
\begin{equation}\label{gizmo}
\mathbb{P}\left(Q(j) \leq \theta j\right) \leq e^{\lambda \theta j} \left((3e^{-2})^{j} + e^{-4Cnr_n^2}\right)
\end{equation}
for all~\(1 \leq j \leq L_0 = \frac{3\epsilon_1 nr_n^2}{8}.\) Since~\(3e^{-2} < 1\) strictly, we can choose~\(C > 0\) smaller if necessary so that~\[(3e^{-2})^{L_0} + e^{-4Cnr_n^2} = (3e^{-2})^{3\epsilon_1nr_n^2/8} + e^{-4Cnr_n^2} \leq e^{-2Cnr_n^2}\] and so setting~\(j=L_0\) in~(\ref{gizmo}) and recalling that~\(\lambda = 3\) we see that
\begin{equation}\label{gizmo_two}
\mathbb{P}\left(Q(L_0)\leq \theta L_0 \right) \leq e^{\lambda \theta L_0} e^{-2Cnr_n^2} = e^{3\theta L_0} e^{-2Cnr_n^2} \leq e^{-Cnr_n^2},
\end{equation}
provided we choose~\(\theta > 0\) small enough.

We recall from the discussion prior to~(\ref{blue_zero}) that~\(Q(L_0) = Q(L_0,i)\) is a lower bound on the number of edges in the path~\({\cal P}_i,\) each of whose weight is at least~\(z_n.\)  Therefore if~\(Q(L_0,i) \geq \theta L_0,\)  then the weight of the path~\({\cal P}_i\) (contained within the square~\(R_i\)) is at least~\[\theta L_0 z_n \geq  \theta \frac{3\epsilon_1 nr_n^2}{8} \cdot z_n.\] To use this bound for estimating the overall spanning tree of the whole graph~\(G,\) we recall the event~\(E_{tot}\) described prior to~(\ref{e_tot_est}). From~(\ref{gizmo_two}) and the union bound, we see that the joint event~\[E_{overall} :=E_{tot} \bigcap  \bigcap_{1 \leq i \leq N} \{Q(L_0,i) \geq \theta L_0\}\] occurs with probability
\begin{align}\label{jyothi_tits_two}
\mathbb{P}(E_{tot} \cap E_{wt}) &\geq 1-Ne^{-Cnr_n^2}-e^{-2C_1nr_n^2p_n} \nonumber\\
&\geq 1-e^{-C_1nr_n^2p_n}
\end{align}
for some constant~\(C_1 > 0,\) provided~\(nr_n^2p_n \geq M\log{n}\) for some large enough~\(M;\) indeed, the final estimate in~(\ref{jyothi_tits_two}) is true since~\[N= \frac{16}{r_n^2} \leq \frac{16np_n}{M\log{n}}\] by Theorem statement and so~\(N  \leq n\) provided~\( M > 0\) is large enough.

We fix such an~\(M\) henceforth and suppose that~\(E_{overall}\) occurs. From the discussion in the previous paragraph we recall that the weight of the path~\({\cal P}_i\) is at least~\(C_2 nr_n^2 z_n\) for some constant~\(C_2 > 0.\) Because~\(E_{con}(i) \supset E_{tot}\) also occurs, the graph~\(G_i\) is connected and so~\({\cal P}_i\) is  in fact a subset of a spanning tree of~\(G_i.\) This  in turn implies that the maximum weight~\(\tau_n^{(i)}\) of a spanning tree of~\(G_i\) satisfies~\[\tau_n^{(i)} \geq C_2 nr_n^2 z_n\] for each~\(1 \leq i \leq N = \frac{16}{r_n^2}.\) Substituting this into~(\ref{neha_tits}), we then get that
\begin{equation}\label{tau_nn_ax}
\tau_n \geq 16C_2 nz_n =  8C_2n H\left(\frac{3\epsilon_1 nr_n^2p_n}{8}\right).
\end{equation}
Finally, combining~(\ref{tau_nn_ax}) with~(\ref{jyothi_tits_two})  we obtain the lower deviation bound in~(\ref{low_dev}) and this completes the proof of Theorem~\ref{thm_mst_max}\((a).\)~\(\qed\)



\emph{Proof of Theorem~\ref{thm_mst_max}\((b)\)}: We obtain the upper bound for~\(\mathbb{E}\tau_n\) by segmenting the weights into distinct bins and explicitly counting the number of edges in each bin. We begin with some preliminary computations regarding the number of neighbours per vertex in~\(G.\) For~\(1 \leq i \leq n\) let~\(N(i)\) be the number of neighbours of the vertex~\(i\) in the graph~\(G.\) Given the vertex location~\(X_i = x,\) a vertex~\(j\) is a neighbour of~\(i\) if and only if~\(X_j \in B(x,r_n),\) the ball of radius~\(r_n\) centred at~\(x\) and the edge state~\(Z(i,j)=1.\) Thus~\(j\) is a neighbour of~\(i\) with probability
\begin{align}
&\mathbb{P}(\{X_i \in B(x,r_n)\} \bigcap \{Z(i,j)=1\}) \nonumber\\
&\;\;\;=\;\;\mathbb{P}(X_j \in B(x,r_n)) \mathbb{P}(Z_{i,j}=1) \nonumber\\
&\;\;\;=\;\;p_n\int_{B(x,r_n)}f \nonumber
\end{align}
where we recall that~\(f\) is the common density of the vertex locations. Thus~\(N(i)\) is Binomially distributed with parameters~\(n-1\) and~\(p_n \int_{B(x,r_n)}f.\)

For any point~\(x\) in the unit square~\(S,\) the area of~\(B(x,r_n)\) is at least~\[\frac{\pi x^2}{4} \geq \frac{x^2}{2}\] and at most~\(\pi x^2 \leq 4x^2.\) Therefore using the density bounds in~(\ref{f_eq}), we get that
\[ \frac{\epsilon_1 r_n^2p_n}{2}  \leq  p_n\int_{B(x,r_n)} f \leq 4\epsilon_2 r_n^2p_n\] and so defining
\[E_{nei}(i) := \left\{\frac{\epsilon_1 nr_n^2p_n}{4} \leq N(i) \leq 8\epsilon_2 nr_n^2p_n\right\},\] we get from the standard deviation estimate~(\ref{conc_est_f}) that
\[\mathbb{P}\left(  E_{nei}(i) \mid X_i=x\right) \geq 1- e^{-2C_1nr_n^2p_n} \]
for some constant~\(C_1 > 0\) not depending on the choice of~\(i\) or~\(x.\) Averaging over~\(X_i\) and using Fubini's theorem, we get
\[\mathbb{P}\left(E_{nei}(i)\right) \geq 1- e^{-2C_1nr_n^2p_n}\] and so defining~\[E_{nei} := \bigcap_{i=1}^{n} E_{nei}(i)\] and applying the union bound we obtain
\begin{equation}\label{e_nei_est}
\mathbb{P}(E_{nei}) \geq 1-ne^{-2C_1nr_n^2p_n} \geq 1- e^{-C_1nr_n^2p_n}
\end{equation}
for all~\(n\) large, provided~\(nr_n^2p_n \geq M\log{n}\) and~\(M\) is a large enough constant. We fix such an~\(M\) henceforth.

Letting~\[\mathbb{P}_{nei}(.) := \mathbb{P}(. \mid E_{nei})\] be the probability measure conditioned on the occurrence of~\(E_{nei},\) we obtain an upper bound for~\(\mathbb{E}\tau_n\) as follows. Let~\({\cal T}_n\) be the maximum weight spanning tree of the largest component of~\(G.\) The total weight of edges in~\({\cal T}_n\) each having weight at most~\(2H(nr_n^2p_n) =: 2y_n\) is at most~\(2(n-1)y_n \leq 2ny_n.\)

Next for integer~\(j \geq 1\) say that an edge~\(f \in G\) is~\(j-\)\emph{bad} if the edge weight~\[W(f)  \in [2jy_n,2(j+1)y_n)\] and let~\(N_{bad}(j)\) be the number of~\(j-\)bad edges of~\(G.\) The total weight of~\(j-\)bad edges is at most~\(2(j+1)y_n \leq 4jy_n\) and so we get that
\begin{equation}\label{tau_later}
\tau_n \leq 2ny_n + 4y_n\sum_{j \geq 1}jN_{bad}(j).
\end{equation}
Taking expectations with respect to~\(\mathbb{P}_{nei},\) we obtain the upper bound
\begin{equation}\label{e_tau_n_up}
\mathbb{E}_{nei}\tau_n \leq 2ny_n + 4y_n \sum_{j \geq 1}j\mathbb{E}_{nei}(N_{bad}(j).
\end{equation}

Using the scaling relation~(\ref{dilpax}) in Theorem statement with~\(a=j\) and~\(x = y_n\) we obtain for any edge~\(f \in K_n\) that
\begin{align}\label{betty_gilpin}
\mathbb{P}_{nei}(W(f) \geq 2jy_n) &\leq \frac{C}{j^{s}} \mathbb{P}(W(f) \geq 2y_n) \nonumber\\
&\leq \frac{C}{j^{s}} \cdot \frac{1}{nr_n^2p_n},
\end{align}
by definition of~\(H(.)\) in~(\ref{h_def}). The occurrence of~\(E_{nei}\) implies that there are at most~\(8\epsilon_2nr_n^2p_n\) neighbours for each vertex and since the sum of vertex degrees is twice the number of edges, we see that~\(G\) has at most~\(8\epsilon_2n^2r_n^2p_n\) edges. From~(\ref{betty_gilpin}), we then get that
\begin{equation}\label{e_nei_bad}
\mathbb{E}_{nei}(N_{bad}(j)) \leq 8\epsilon_2n^2r_n^2p_n \cdot \frac{C}{j^{s}} \cdot \frac{1}{nr_n^2p_n} = \frac{8C\epsilon_2n}{j^{s}}.
\end{equation} Plugging this into~(\ref{e_tau_n_up}), we get
\begin{align}
\mathbb{E}_{nei}\tau_n &\leq 2ny_n + 4y_n\sum_{j \geq 1} j \cdot \frac{8C\epsilon_2n}{j^{s}} \nonumber\\
&= 2ny_n + 32C\epsilon_2 ny_n \sum_{j \geq 1}\frac{1}{j^{s-1}} \nonumber\\
&\leq 2 ny_n + D_1 ny_n \label{one_part}
\end{align}
for some constant~\(D_1 > 0,\) since~\(s > 2\) strictly.

From~(\ref{one_part}) and the estimate~(\ref{e_nei_est}) for~\(E_{nei},\) we therefore get that
\begin{align}
\mathbb{E}\tau_n &= \mathbb{E}(\tau_n \mid E_{nei})\mathbb{P}(E_{nei}) + \mathbb{E}(\tau_n \mid E^c_{nei}) \mathbb{P}(E^c_{nei}) \nonumber\\
&\leq  \mathbb{E}(\tau_n \mid E_{nei}) + \mathbb{E}(\tau_n \mid E^c_{nei}) e^{-C_1 nr_n^2p_n} \nonumber\\
&= \mathbb{E}_{nei}\tau_n + \mathbb{E}(\tau_n \mid E^c_{nei}) e^{-C_1 nr_n^2p_n} \nonumber\\
&\leq (D_1+2)ny_n +  \mathbb{E}(\tau_n \mid E^c_{nei}) e^{-C_1 nr_n^2p_n} \label{e_tau_n_tot}
\end{align}
by~(\ref{one_part}). If the event~\(E^c_{nei}\) occurs, then we use the trivial upper bound~\[\tau_n \leq \sum_{f \in K_n} W(f),\] the sum of weights of all edges of the complete graph~\(K_n\) to get that
\begin{align}
\mathbb{E}( \tau_n \mid E_{nei}^c) &\leq \mathbb{E}\left(\sum_{f \in K_n} W(f) \mid E_{nei}^c\right) \nonumber\\
&= \mathbb{E} \sum_{f \in K_n} W(f) \label{inter_med}
\end{align}
since the event~\(E_{nei}\) depends only on the vertex locations and is therefore independent of the  edge weights.

Using the scaling relation~(\ref{dilpax})  with~\(x = 2x_0,\) we see that \[\mathbb{P}(W(f) \geq 2x_0a) \leq \frac{C}{a^s}\] for all~\(a > 1\) and so the edge weight~\(W(f)\) has finite expectation. Since there are~\({n \choose 2} \leq n^2\) edges in~\(K_n\) the final expression in~(\ref{inter_med}) is at most~\(D_2n^2\) for some constant~\(D_2 > 0\) and plugging this into~(\ref{e_tau_n_tot}), we get
\[\mathbb{E}\tau_n \leq (D_1+2)ny_n + D_2n^2e^{-C_1nr_n^2p_n}.\] Since the edge weights are bounded away from zero, we know that~\(y_n \geq \epsilon_0\) for some constant~\(\epsilon_0 > 0\) and so if~\(nr_n^2p_n \geq M\log{n}\) for some large enough constant~\(M > 0\) then we see that
\begin{equation}\label{e_tau_n_ax}
\mathbb{E}\tau_n \leq (D_1+2)ny_n + D_2n^2 \cdot \frac{1}{n^2} \leq D_3 ny_n = D_3 nH(nr_n^2p_n)
\end{equation}
for some constant~\(D_3 > 0.\) This obtains the desired expectation upper bound for~\(\mathbb{E}\tau_n.\)

We obtain the upper deviation bound for~\(\tau_n\) as follows: If~\(E_{nei}\) occurs, then as argued in the discussion following~(\ref{betty_gilpin}), we see that the number of edges in~\(G\) is at most~\(8\epsilon_2n^2r_n^2p_n.\) Therefore from~(\ref{betty_gilpin}), we see that the number of~\(j-\)bad edges~\(N_{bad}(j)\) is stochastically dominated from above by a Binomial random variable with parameters~\(8\epsilon_2n^2r_n^2p_n\) and~\(\frac{C}{j^{s}} \cdot \frac{1}{nr_n^2p_n}.\) Using the deviation estimate~(\ref{conc_est_f}), we then get that
\[\mathbb{P}_{nei}\left(N_{bad}(j) \geq \frac{C_1 n}{j^{s}}\right) \leq \exp\left(-\frac{2C_2n}{j^{s}}\right)\]
for some constants~\(C_1,C_2 > 0,\) not depending on the choice of~\(j.\)

If~\( j \leq \left(\frac{n}{(\log{n})^2}\right)^{1/s} =: w_n\) then
\[\mathbb{P}_{nei}\left(N_{bad}(j) \geq \frac{C_1n}{j^s}\right) \leq e^{-2C_2(\log{n})^2}\] and so defining
\[E_{low} := \bigcap_{j=1}^{w_n} \left\{N_{bad}(j) \leq \frac{C_1n}{j^{s}}\right\},\] we get from the union bound that
\begin{equation}\label{e_low_est}
\mathbb{P}_{nei}(E_{low}) \geq 1-w_n \cdot e^{-2C_2(\log{n})^2} \geq 1-e^{-C_2(\log{n})^2}
\end{equation}
for all~\(n\) large since~\(w_n \leq n.\) If~\(E_{low}\) occurs, then recalling the notation~\(y_n = H(nr_n^2p_n),\) we see that the total weight of edges having weight at most~\(2w_ny_n\) is at most
\begin{equation}\label{j_bad_one}
\sum_{j=1}^{w_n} \frac{C_1n}{j{^s}} \cdot 2(j+1)y_n \leq 4C_1ny_n \sum_{j \geq 1} \frac{1}{j^{s-1}} \leq C_3ny_n,
\end{equation}
for some constant~\(C_3 > 0.\)

If~\(j > w_n,\) then the weight of a~\(j-\)bad edge is at least~\(2w_ny_n\) and we define such an edge to be \emph{heavy}. Again using the scaling relation~(\ref{dilpax}) with~\(a = w_n\) we get that
\begin{align}
\mathbb{P}_{nei}(\text{edge }f \text{ is heavy}) &\leq \frac{C(\log{n})^2}{n} \mathbb{P}_{nei}(W(f) \geq 2y_n) \nonumber\\
&= \frac{C(\log{n})^2}{n} \mathbb{P}_{nei}\left(W(f)  \geq 2H(nr_n^2p_n)\right) \nonumber\\
&\leq \frac{C(\log{n})^2}{n} \cdot \frac{1}{nr_n^2p_n}
\end{align}
by definition of~\(H(.)\) in~(\ref{h_def})).

Since~\(E_{nei}\) occurs the total number of edges in~\(G\) is at most~\(8\epsilon_2n^2r_n^2p_n\) (see discussion following~(\ref{betty_gilpin})) and the number~\(N_{heavy}\) of heavy edges is stochastically dominated from above by a Binomial random variable~\(Z\) with parameters~\(8\epsilon_2n^2r_n^2p_n\) and~\(\frac{C(\log{n})^2}{n^2r_n^2p_n}.\) Consequently,
\[\mathbb{E}e^{Z} = \left(1+\frac{C(\log{n})^2(e-1)}{n^2r_n^2p_n}\right)^{8\epsilon_2n^2r_n^2p_n} \leq e^{C_4(\log{n})^2}\]
for some constant~\(C_4 > 0\) and so applying the  Chernoff bound we get that
\begin{align}
\mathbb{P}_{nei}(N_{heavy} \geq 2(\log{n})^3) &\leq \mathbb{P}(Z \geq 2(\log{n})^3) \nonumber\\
&\leq e^{-2(\log{n})^3} \mathbb{E}e^{Z} \nonumber\\
&\leq e^{-2(\log{n})^3} e^{C_4(\log{n})^2} \nonumber\\
&\leq e^{-(\log{n})^3} \label{heavy_est}
\end{align}
for all~\(n\) large.

Thus with high probability the number of heavy edges is at most~\(2(\log{n})^3.\) To  estimate the total weight of heavy edges, we use one final estimate regarding ``super-heavy" edges. For~\(\epsilon > 0\) small to be determined later, let~\(E_{good}\) be the event that each edge of~\(G\) has weight at most~\(2H((n^2r_n^2p_n)^{1+\epsilon}) =: 2z_n.\) By definition, for any edge~\(f\) we have
\[\mathbb{P}_{nei}\left(W(f) \geq 2z_n\right) \leq \frac{1}{(n^2r_n^2p_n)^{1+\epsilon}}\] and since there are at most~\(8\epsilon_2n^2r_n^2p_n\) edges in~\(G\) (see discussion following~(\ref{betty_gilpin})), we get from the union bound that
\begin{equation}\label{e_good_est}
\mathbb{P}_{nei}(E_{good}) \geq 1- \frac{8\epsilon_2n^2r_n^2p_n}{(n^2r_n^2p_n)^{1+\epsilon}} \geq 1-\frac{C_5}{(n^2r_n^2p_n)^{\epsilon}}
\end{equation}
for some constant~\(C_5 > 0.\)

Recalling the event~\(E_{tot}\) in the proof of part~\((a)\) above and defining~\[E_{net} := E_{tot} \cap E_{low} \bigcap \{N_{heavy} \leq 2(\log{n})^3\} \bigcap E_{good},\] we get from the respective estimates~(\ref{e_tot_est}),~(\ref{e_low_est}),~(\ref{heavy_est}) and~(\ref{e_good_est}) that
\begin{align}
\mathbb{P}_{nei}(E_{net}) &\geq 1 - e^{-C_3nr_n^2p_n} - e^{-C_2(\log{n})^2}- e^{-(\log{n})^3} - \frac{C_5}{(n^2r_n^2p_n)^{\epsilon}} \nonumber\\
&\geq 1- \frac{2C_5}{(n^2r_n^2p_n)^{\epsilon}} \nonumber
\end{align}
for all~\(n\) large, since~\(nr_n^2p_n \geq M\log{n}\) by Theorem statement. Combining with the estimate for~\(E_{nei}\) in~(\ref{e_nei_est}), we get that
\begin{align}
\mathbb{P}(E_{nei} \cap E_{net}) &\geq 1- \frac{2C_5}{(n^2r_n^2p_n)^{\epsilon}} - e^{-C_1nr_n^2p_n} \nonumber\\
&\geq 1-\frac{3C_5}{(n^2r_n^2p_n)^{\epsilon}}
\end{align}
again since~\(nr_n^2p_n \geq M\log{n}\) by Theorem statement.

Suppose~\(E_{nei} \cap E_{net}\) occurs so that the graph~\(G\) is connected (see discussion prior to~(\ref{neha_tits})) and let~\({\cal T}_n\) be the maximum weight of a spanning tree of~\(G.\) The total weight of edges in~\({\cal T}_n\) each having weight at most~\(2y_n\) is at most~\(2ny_n.\) Since~\(E_{low}\) occurs, the total weight of edges having weight at most~\(2w_ny_n\) is at most~\(C_3ny_n\) (see discussion prior to~(\ref{j_bad_one})). Finally,  the number of heavy edges is at most~\( 2(\log{n})^3\) and since~\(E_{good}\) occurs, the weight of any edge in~\(G\) is at most~\(2z_n.\) Thus the total weight of heavy edges is at most~\(4(\log{n})^3z_n.\) Combining we see that the total weight of~\({\cal T}_n\) is at most~\((2+C_3)n y_n + 4(\log{n})^3 z_n.\)

In what follows, we show that~\((\log{n})^3z_n \leq ny_n\) and this would obtain the desired upper deviation estimate for~\(\tau_n\) in the statement of the Theorem. Using the scaling relation~(\ref{dilpax})  we get for any edge~\(f\) that
\[\mathbb{P}\left(W(f) \geq C^{1/s} \cdot (n^2r_n^2p_n)^{(1+\epsilon)/s}\right) \leq \frac{1}{(n^2r_n^2p_n)^{1+\epsilon}}\]
and so~\[z_n = H((n^2r_n^2p_n)^{1+\epsilon}) \leq (n^2r_n^2p_n)^{(1+\epsilon)/s} \leq n^{2(1+\epsilon)/s}.\]

Since~\(s > 2\) strictly by Theorem statement, we can  choose~\(\epsilon > 0\) sufficiently small so that~\(2(1+\epsilon) < s\) and get that
\begin{equation}\label{jelly}
\frac{(\log{n})^3}{n} z_n \leq \frac{(\log{n})^3}{n} \cdot n^{2(1+\epsilon)/s} \longrightarrow 0
\end{equation}
as~\(n \rightarrow \infty.\) On the other hand, using the relation~\(nr_n^2p_n \geq M\log{n}\) from the Theorem statement, we see that~\[\mathbb{P}\left(W(f) \geq 2H(nr_n^2p_n) \right) \leq \frac{1}{nr_n^2p_n} \leq \frac{1}{2} \leq \mathbb{P}\left(W(f) \geq \frac{H(2)}{2}\right)\] for all~\(n\) large and so from~(\ref{jelly}), we then get that~\[2y_n = 2H(nr_n^2p_n) \geq \frac{H(2)}{2} \geq \frac{2(\log{n})^3}{n} \cdot z_n\] for all~\(n\) large. This completes the proof of the Theorem.~\(\qed\)





Recalling the definition of the function~\(H(.)\) in~(\ref{h_def}), we have the following variance bound for the maximum weight of a spanning tree.
\begin{theorem}\label{thm_var_est} Let~\(\epsilon_1,\epsilon_2\) be as in~(\ref{f_eq}) and suppose that~\(nr_n^2p_n \geq M\log{n}\) for some constant~\(M > 0.\) Also suppose that the ccdf scaling condition~(\ref{dilpax}) in Theorem~\ref{thm_mst_max}\((b)\) holds for some constant~\(s > 3\) strictly. There is a constant~\(\gamma> 0\) such that for any~\(M > \gamma,\)~\(L = L(n) >0\) and~\(2 < s_1 < s,\) the variance~\[var(\tau_n) \leq  \gamma n (\sigma^2_{loc} + \sigma^2_{wt}),\] where
\begin{equation}\label{sigma_loc_exp}
\sigma^2_{loc} :=  (nr_n^2p_n)^2 \left(L^2 + n^{4}r_n^2p_n F_c(L) + n^2F_c^{1-2/s_1}(L)\right) + n^{6}e^{-\gamma^{-1}nr_n^2p_n}
\end{equation} is the scaled contribution due to randomness in the vertex locations and \[\sigma^2_{wt} := H^2(nr_n^2p_n)\] is the scaled contribution due to randomness in edge states and weights.
\end{theorem}
We have obtained the above variance bound under the stronger condition that~\(nr_n^2p_n^2\) is at least of the order of~\(\log{n}\) (as opposed to the bounds in Theorem~\ref{thm_mst_max} that hold under the weaker condition~\(nr_n^2p_n \geq M\log{n}\)). Also, we have expressed~\(\sigma_{loc}^2\) in terms of the parameter~\(L\) that could be chosen according to convenience and we illustrate this aspect in our examples below.

Before we do so, we have the following remark. If suppose we choose~\(L \geq \epsilon\) a constant, then the first term~\((nr_n^2p_n)^2L^2\) in~(\ref{sigma_loc_exp}) is at least~\((nr_n^2p_n)^2\epsilon^2.\) On the other hand because of the ccdf scaling relation~(\ref{dilpax}), we know (see discussion following power law example in Theorem~\ref{thm_mst_max}) that~\(H(z) \leq D z^{1/s}\) for some constant~\(D > 0\) and all~\(z\) large. Since~\(nr_n^2p_n\) is at least of the order of~\(\log{n},\) this implies that~\(\sigma^2_{wt}\) is always much smaller than~\(\sigma_{loc}^2.\) However, in Theorem~\ref{thm_var_est} above, we have provided the scaled variance contributions due to random vertex locations and random weights/states  separately for clarity and in fact from the proof  below, we see that if the vertex locations are \emph{fixed}, then the variance of~\(\tau_n\) is simply upper bounded by a constant multiple of~\(n\sigma_{wt}^2.\)


As before, we  consider the two examples described following Theorem~\ref{thm_mst_max}.

\underline{\emph{Power Law}}: Recalling the notation in the paragraph containing~(\ref{tau_n_power_law}), we see that~\(F_c(L) \leq \frac{A_2}{L^s}\) and that~\(H(nr_n^2p_n) \leq B_2 (nr_n^2p_n)^{1/s}\) for constants~\(A_2,B_2 >0.\) Therefore using~\(r_n,p_n \leq 1\) and choosing~\(3 < s_1 < s,\) we get from~(\ref{sigma_loc_exp}) that
\[\sigma_{loc}^2 \leq (nr_n^2p_n)^2 \left(L^2 + \frac{n^4}{L^{s}} + \frac{n^2}{L^{s/3}} \right) + n^{6}e^{-\gamma^{-1}nr_n^2p_n}.\]
Setting~\(L = n^{\max\left(\frac{4}{2+s},\frac{6}{6+s}\right)},\) we see that~\(\frac{n^{4}}{L^s} \leq L^2\) and~\(\frac{n^2}{L^{s/3}} \leq L^2\) and so~\(\sigma_{loc}^2 \leq 4 (nr_n^2p_nL)^2\) for all~\(n\) large. Since~\(L \geq 1\) the discussion above implies that~\(\sigma_{loc}^2\) is much larger than~\(\sigma_{wt}^2\) and so \[var(\tau_n) \leq Dn(\sigma_{loc}^2 + \sigma_{wt}^2) \leq 2Dn\sigma_{loc}^2 \leq 8Dn \cdot (nr_n^2p_nL)^2\] for all~\(n\) large and some constant~\(D > 0.\) On the other hand from~(\ref{tau_n_power_law}), we see that~\(\mathbb{E}\tau_n \geq D_1 n \left(nr_n^2p_n\right)^{1/s}\) for some constant~\( D_1 > 0\) and so
\begin{equation}\label{l2_conv_power}
\mathbb{E}\left(\frac{\tau_n}{\mathbb{E}\tau_n}-1\right)^2 \leq \frac{D_2}{n} \cdot (nr_n^2p_n)^{2-2/s} \cdot L^2
\end{equation}
for some constant~\(D_2 > 0.\)

Suppose now that~\(r_n = \frac{1}{n^{\theta}},p_n = \frac{1}{n^{\beta}}\) for some constants~\(0 < \theta,\beta < 1\) satisfying~\(2\theta+2\beta < 1\) so that the condition~\(nr_n^2p^2_n \geq M\log{n}\) is satisfied. From~(\ref{l2_conv_power}), we see that if~\((nr_n^2p_n)^{2-2/s}L^2\) is much smaller than~\(n\) or equivalently if
\[(1-2\theta-\beta)\left(2-\frac{2}{s}\right) + 2\max\left(\frac{4}{2+s},\frac{6}{6+s}\right) < 1\] strictly, then~\(\frac{\tau_n}{\mathbb{E}\tau_n} \longrightarrow 1\) in~\(L^2\) as~\(n \rightarrow \infty.\) In words, if~\(G\) is sparse enough (i.e.,~\(\theta\) close to~\(\frac{1}{2}\) or~\(\beta\) close   to~\(1\)) and the edge weights are light enough (i.e.,~\(s\) is large), then we are guaranteed~\(L^2-\)convergence of the maximum weight of a spanning tree, appropriately scaled and centred.

\underline{\emph{Exponential}}: Recalling the notation in the paragraph containing~(\ref{tau_n_exponential}), we see that~\(F_c(L) \leq e^{-A_2L^{\alpha}}\) and that~\(H(nr_n^2p_n) \leq B_2 \left(\log(nr_n^2p_n)\right)^{1/\alpha}\) for constants~\(A_2,B_2 >0.\) Therefore again using~\(r_n,p_n \leq 1\) and choosing~\(L = (\log{n})^{\zeta}\) for some large enough constant~\(\zeta > 0,\) we get from~(\ref{sigma_loc_exp}) that
\begin{align}
var(\tau_n) &\leq Dn(nr_n^2p_n)^2 \left((\log{n})^{2\zeta} + n^4e^{-D_1(\log{n})^2} + n^2e^{-D_1(\log{n})^2}\right) \nonumber\\
&\leq 2Dn(nr_n^2p_n)^{2} \cdot (\log{n})^{2\zeta} \label{tau_expo_est}
\end{align}
for some constants~\(D,D_1 > 0\) and all~\(n\) large.

From~(\ref{tau_n_exponential}), we know that~\(\mathbb{E}\tau_n \geq D_2 n \left(\log(nr_n^2p_n)\right)^{1/\alpha} \geq D_2n\) provided~\(nr_n^2p_n \geq M\log{n}\) for large enough~\(M\) and so we get from~(\ref{tau_expo_est}) that
\begin{equation}\label{l2_conv_expo}
\mathbb{E}\left(\frac{\tau_n}{\mathbb{E}\tau_n}-1\right)^2 \leq \frac{D_3}{n} \cdot (nr_n^2p_n)^{2} \cdot (\log{n})^{2\zeta}
\end{equation}
for some constant~\(D_3 > 0.\) As before, suppose now that~\(r_n = \frac{1}{n^{\theta}}\) and~\(p_n = \frac{1}{n^{\beta}}\) where~\(0 < 2\theta + 2\beta < 1\) so that the stronger condition~\(nr_n^2p_n^2 \geq M\log{n}\) is satisfied. If in addition~\(nr_n^2p_n\) is much smaller than~\(\sqrt{n}\) or equivalently if~\(\theta\) satisfies~\[\frac{1}{2}\left(\frac{1}{2}-\beta\right) < \theta < \frac{1}{2}-\beta\] then from~(\ref{l2_conv_expo}), we obtain~\(L^2-\)convergence of~\(\frac{\tau_n}{\mathbb{E}\tau_n}.\)

\emph{Proof of Theorem~\ref{thm_var_est}}: For~\(1 \leq j \leq n\) let~\({\cal F}_j\) be the sigma-field generated by the vertex locations~\(\{X_l\}_{1 \leq l \leq j}.\) We also  deterministically order the edges of~\(K_n\) as~\(f_1,\ldots,f_m, m = {n \choose 2}\) and for~\(n+1 \leq j \leq n+m\) we let~\({\cal F}_j\) be the sigma-field generated by~\(\{X_l\}_{1 \leq l \leq n} \bigcup \bigcup_{l=1}^{j-n}\{(Z(f_l),W(f_l))\}\) where we recall that~\(Z(f_l)\) and~\(W(f_l)\) respectively denote the state and weight of the edge~\(f_l.\)

From the martingale difference property we then get
\begin{equation}\label{var_est_am}
var(\tau_n) = \sum_{j=1}^{n+m} \mathbb{E}\left(\mathbb{E}\left(\tau_n \mid {\cal F}_j\right) - \mathbb{E}\left(\tau_n \mid {\cal F}_{j-1}\right)\right)^2
\end{equation}
and we rewrite the above expression in a more convenient form as follows: For~\(1 \leq j \leq n\) we define~\(\tau_n^{(j)}\) as the maximum weight of a spanning tree of the largest component in the random graph obtained by replacing the vertex location~\(X_j\) with an independent copy~\(X_j^{(c)}.\) Similarly for~\(n+1 \leq j \leq n+m,\) the term~\(\tau_n^{(j)}\) is the maximum weight of a spanning tree of the largest component in the random graph obtained by replacing the state and weight~\((Z(f_{j-n}),W(f_{j-n}))\) of the edge~\(f_{j-n}\) with an independent copy~\((Z^{(c)}(f_{j-n}),W^{(c)}(f_{j-n})).\)

With the above notations, we get
\begin{align}
\left(\mathbb{E}(\tau_n \mid {\cal F}_j) - \mathbb{E}(\tau_n \mid {\cal F}_{j-1})\right)^2 &= \left(\mathbb{E}(\tau_n-\tau_n^{(j)} \mid {\cal F}_j)\right)^2 \nonumber\\
&\leq \mathbb{E}\left((\tau_n - \tau^{(j)}_n)^2 \mid{\cal F}_j\right) \nonumber
\end{align}
by the Jensen's inequality for conditional expectations. Plugging this into~(\ref{var_est_am}), we then get
\begin{align}
var(\tau_n) &\leq \sum_{j=1}^{n+m} \mathbb{E}\left(\tau_n-\tau_n^{(j)}\right)^2 \nonumber\\
&= \sum_{j=1}^{n} \mathbb{E}\left(\tau_n - \tau_n^{(j)}\right)^2 + \sum_{j=n+1}^{n+m} \mathbb{E}\left(\tau_n - \tau_n^{(j)}\right)^2 \nonumber\\
&= n I_{loc} + m I_{wt} \label{var_est_am_two}
\end{align}
by symmetry, where~\(I_{loc} := \mathbb{E}\left(\tau_n - \tau_n^{(1)}\right)^2\) and~\(I_{wt} := \mathbb{E}\left(\tau_n - \tau_n^{(1)}\right)^2\) respectively could be interpreted as the respective scaled contributions due to randomness in vertex locations and the edge states and weights.

We now proceed in two steps: In the first step, we estimate~\(I_{loc}\) and in the second step we bound~\(I_{wt}.\) Finally, we combine the two expressions obtained to establish the variance bound in the statement of the Theorem.

\emph{\underline{Step 1}}: To evaluate~\(I_{loc},\) we recall from the proof of Theorem~\ref{thm_mst_max}\((a),\) the subgraph~\(G({\cal V}) \subset G\) obtained after removing the vertices in~\({\cal V}.\) Let~\(\tau_{rem}\) be the maximum weight of the spanning tree of the largest component of the graph~\(G(\{1\})\) obtained after removing the vertex~\(1.\) By the triangle inequality we have
\[|\tau_n-\tau_n^{(1)}| \leq |\tau_n-\tau_{rem}| + |\tau_{rem}-\tau_{n}^{(1)}|\] and so using~\((a+b)^2 \leq 2(a^2+ b^2),\) we get
\begin{align}
I_{loc} &= \mathbb{E}\left(\tau_n-\tau_n^{(1)}\right)^2  \nonumber\\
&\leq 2\mathbb{E}\left(\tau_n-\tau_{rem}\right)^2 + 2\mathbb{E}\left(\tau_{rem}-\tau_n^{(1)}\right)^2 \nonumber\\
&= 4 \mathbb{E}\left(\tau_n - \tau_{rem}\right)^2, \label{i_loc_est_one}
\end{align}
by symmetry.

We estimate~\(\tau_n-\tau_{rem}\) as follows. For~\(L > 0\) we first let~\(E_{wt}(L)\) be the event that each edge of~\(G\) has weight at most~\(L.\) Next, as in Figure~\ref{fig_squares}, we divide the unit square~\(S\) into small~\(\frac{r_n}{4} \times \frac{r_n}{4}\) squares~\(\{R_i\}_{1 \leq i \leq N}\) where~\(N = \frac{16}{r_n^2}\) is assumed to be an integer (see the discussion in the first paragraph of the proof of Theorem~\ref{thm_mst_max}\((a)\)). Recalling the events~\(E_{tot}({\cal V})\) and~\(E_{nei}\) defined in the proof of Theorems~\ref{thm_mst_max}\((a)\) and~\((b),\) respectively, we assume henceforth that~\[E_{all} := E_{wt}(L) \cap E_{tot}(\{1\}) \cap E_{tot}(\emptyset) \cap E_{nei}\] occurs so that:\\
\((i)\) Both the graphs~\(G(\{1\})\) and~\(G\) are connected (see discussion following~(\ref{e_tot_est}) in the proof of Theorem~\ref{thm_mst_max}\((a)\)),\\
\((ii)\) Each vertex has at most~\(Cnr_n^2p_n\) neighbours in~\(G\) for some constant~\(C > 0\)  (see paragraph containing~(\ref{e_nei_est})) and\\
\((iii)\) Every edge in~\(G\) has weight at most~\(L.\)

Because~\(G(\{1\})\) and~\(G\) are both connected, any spanning tree of~\(G(\{1\})\) can be extended to a spanning tree of~\(G.\) The overall weight of a tree only increases upon adding more edges and so~\(\tau_{rem} \leq \tau_n.\) For the reverse direction, let~\({\cal S}\) be any spanning tree of~\(G\) and let~\(v_1,\ldots,v_t, t \leq Cnr_n^2p_n\) be the neighbours of the vertex~\(1\) in~\(G.\) Removing the vertex~\(1\) from~\(G\) we then get~\(t\) subtrees~\({\cal S}_1,\ldots,{\cal S}_t\) as shown in Figure~\ref{fig_sub_trees}. Here~\(v_1=A,v_2=B\) and~\(v_3=C\) are the neighbours of the vertex~\(1\) in the graph~\(G\) and the subtrees~\({\cal S}_1,{\cal S}_2\) and~\({\cal S}_3\) are denoted by solid triangles.

\begin{figure}[tbp]
\centering
\includegraphics[width=6in, trim= 220 200 50 110, clip=true]{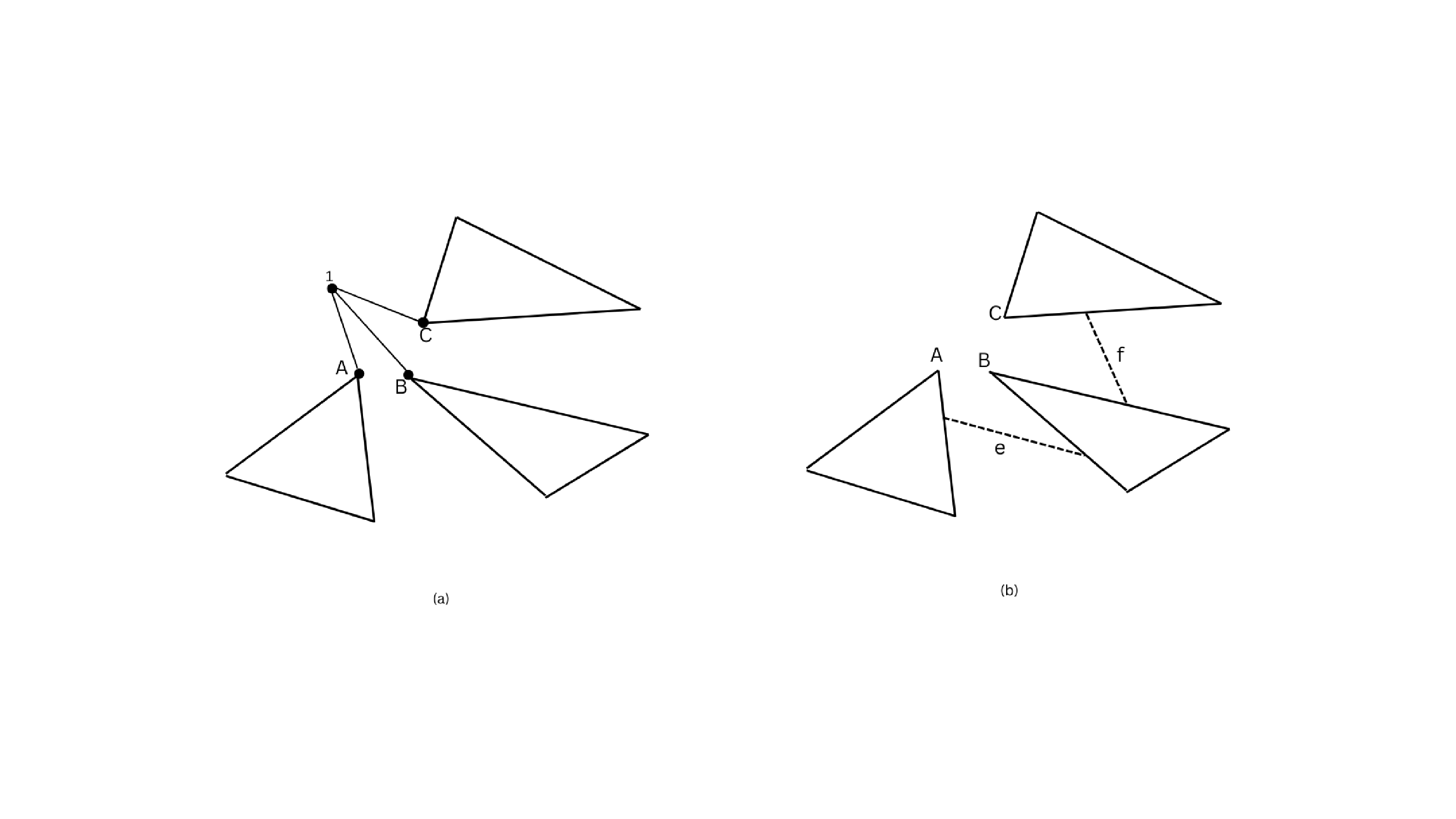}
\caption{The subtrees~\({\cal S}_i, 1 \leq i \leq t=3\) containing the neighbours~\(v_1=A,v_2=B,v_3=C\) of the vertex~\(1\) are shown in~\((a).\) Removing vertex~\(1\) and adding the  edges~\(h_1=e\) and~\(h_2=f\) gives a spanning tree of the graph~\(G(\{1\})\) as shown in~\((b).\)}
\label{fig_sub_trees}
\end{figure}

Since~\(G(\{1\})\) is also connected, we can always find~\(t-1\) edges~\(h_1,\ldots,h_{t-1}\) such that the union~\(\bigcup_{l=1}^{t} \{{\cal S}_l\} \bigcup \bigcup_{l=1}^{t-1} \{h_l\}\) forms a spanning tree~\({\cal S}_{new}\) of~\(G(\{1\})\) (In figure~\ref{fig_sub_trees},~\(h_1=e\) and~\(h_2=f\)). The occurrence of the event~\(E_{wt}(L)\) ensures that the weight of any edge is at most~\(L\) and so the weight~\(W({\cal S}) = \sum_{f \in {\cal S}}W(f)\) of the ``original" tree~\({\cal S}\) satisfies
\begin{align}
W({\cal S}) &\leq W({\cal S}_{new}) + Lt \nonumber\\
&\leq W({\cal S}_{new}) + CLnr_n^2p_n \nonumber\\
&\leq \tau_{rem} + CLnr_n^2p_n. \label{gest_ax}
\end{align}
The estimate~(\ref{gest_ax}) is true for any tree~\({\cal S}\) and so~\(\tau_n \leq \tau_{rem} + CLnr_n^2p_n.\) Combining this with the estimate~\(\tau_{rem} \leq \tau_n\) obtained in the previous paragraph, we get that
\begin{equation}\label{i_loc_one_est}
I_{loc,1} := 4\mathbb{E}\left(\tau_n-\tau_{rem}\right)^2\ind(E_{all}) \leq C^2L^2(nr_n^2p_n)^2.
\end{equation}

It remains to estimate~\[I_{loc,2} := I_{loc}-I_{loc,1} = 4\mathbb{E}\left(\tau_n-\tau_{rem}\right)^2\ind(E^c_{all}).\] Indeed if~\(E^c_{all}\) occurs, then we use the bound~\[\max(\tau_n,\tau_{rem}) \leq \sum_{f \in G}W(f),\] the total weight of all edges in~\(G,\) to get that
\[I_{loc,2} \leq \mathbb{E} \left(\sum_{f \in G} W(f)\right)^2\ind(E^c_{all}).\] The inequality~\((\sum_{i=1}^{l}a_i)^2 \leq l\sum_{i=1}^{l}a_i^2\) for positive~\(\{a_i\}\) and integer~\(l \geq 1\) implies that~\[\left(\sum_{f \in G}W(f)\right)^2 \leq m_G\sum_{f \in G}W^2(f),\] where~\(m_G\) is the number of edges in~\(G\) and moreover,
\[\ind(E^c_{all}) \leq \ind(E_{all,1}^c) + \ind(E^c_{wt}(L)),\]
where~\[E^c_{all,1} := E^c_{tot}(\{1\}) \cup E^c_{tot}(\emptyset) \cup E^c_{nei}.\]
Combining the above we get
\begin{equation}\label{i_loc_22_up}
I_{loc,2} \leq I_{loc,21} + I_{loc,22}
\end{equation}
where
\[I_{loc,21} := \mathbb{E}m_G \sum_{f \in G}W^2(f) \ind(E_{all,1}^c)\]  and
\[I_{loc,22} := \mathbb{E}m_G \sum_{f \in G}W^2(f)\ind(E^c_{wt}(L)).\]

We evaluate~\(I_{loc,21}\) and~\(I_{loc,22}\) in that order below. The events~\(E_{tot}(.)\) and~\(E_{nei}\) are independent of the edge weights and so
\[I_{loc,21} = \mathbb{E}m_G \ind(E_{all,1}^c) \mathbb{E}\left(\sum_{f \in G} W^2(f) \mid G\right).\] Since the edge weights have bounded second moments (see Theorem statement), the term~\[\mathbb{E}\left(\sum_{f \in G} W^2(f) \mid G\right) \leq Dm_G\] for some constant~\(D > 0\) and so
\begin{align}\label{i_loc_est_21}
I_{loc,21} &\leq D \mathbb{E}m_G^2 \ind(E^c_{all,1}) \nonumber\\
&\leq Dn^4\mathbb{P}(E^c_{all,1}) \nonumber\\
&\leq 3Dn^4e^{-D_1nr_n^2p_n}
\end{align}
for some constant~\(D_1 > 0,\) where the second inequality in~(\ref{i_loc_est_21}) follows from the direct bound~\(m_G \leq {n \choose 2} \leq n^2\) and the final estimate in~(\ref{i_loc_est_21}) is obtained from the respective bounds~(\ref{e_tot_est}) and~(\ref{e_nei_est}) for the events~\(E_{tot}(.)\) and~\(E_{nei},\) together with an application of the union bound.

To estimate~\(I_{loc,22},\) we first split
\begin{equation}\label{i_loc_22_split}
I_{loc,22} = \mathbb{E}m_G \sum_{f \in K_n} (T_1(f) + T_2(f)) \ind(f\in G),
\end{equation}
where~\[T_1(f) := \mathbb{E}\left(W^2(f) \ind(W(f) \leq L)\ind(E^c_{wt}(L)) \mid G \right)\]
and
\[T_2(f) := \mathbb{E}\left(W^2(f) \ind(W(f) > L) \ind(E^c_{wt}(L)) \mid G \right).\]
By definition if~\(W(f) \leq L\) and~\(E^c_{wt}(L)\) occurs, then the event~\(A(f)\) that some edge of~\(G\) other than~\(f\) has weight at least~\(L,\) necessarily occurs. In other words~\[\ind(W(f) \leq L) \ind(E^c_{wt}(L)) \leq \ind(A(f))\] and so
\begin{align}
T_1(f) &\leq \mathbb{E}\left(W^2(f) \ind(A(f)) \mid G\right) \nonumber\\
&= \mathbb{E}W^2(f) \mathbb{P}(A(f) \mid G),  \nonumber\\
&\leq \mathbb{E}W^2(f) m_G F_c(L) \nonumber\\
&\leq D_2 m_GF_c(L)\label{gen_ta}
\end{align}
for some constant~\(D_2 >0,\) where the second relation in~(\ref{gen_ta}) is true since~\(W(f)\) is independent of~\(A(f)\) by definition, the third expression in~(\ref{gen_ta}) is obtained from the union bound and the final estimate in~(\ref{gen_ta}) is true since the edge weights have bounded second moment, by Theorem statement.

To upper bound~\(T_2(f),\) we recall from the ccdf scaling relation in the Theorem statement that~\(F_c(ax_0) \leq \frac{C}{a^s} \cdot F_c(x_0)\) for some~\(s > 2,x_0 > 0\) strictly and all~\(a > 1.\) So for~\(2 < s_1 < s\) strictly, we have that
\begin{align}
\mathbb{E}W^{s_1}(f) &\leq \int_{0}^{\infty} y^{s_1-1} F_c(y) dy \nonumber\\
&=  x_0^{s_1} \int_{0}^{\infty} a^{s_1-1} F_c(ax_0) da, \nonumber
\end{align}
by a change of variable~\(y=ax_0.\) Splitting the integral into two terms~\(\int_{0}^{1} + \int_{1}^{\infty}\) we get that
\[\mathbb{E}W^{s_1}(f) \leq D_3 + D_3 \int_{1}^{\infty} \frac{da}{a^{s-s_1+1}} \leq D_4\]
for some constants~\(D_3,D_4 >0.\) Therefore using the H\"olders inequality the term~\(T_2(f)\) is upper bounded as
\begin{align}
T_2(f) &\leq \mathbb{E}\left(W^2(f)\ind(W(f) > L) \mid G\right) \nonumber\\
&= \mathbb{E}W^2(f) \ind(W(f) > L)  \nonumber\\
&\leq \left(\mathbb{E}W^{s_1}(f)\right)^{2/s_1} F_c^{1-2/s_1}(L) \nonumber\\
&\leq D_5  F^{1-2/s_1}_c(L)\label{t_two_f}
\end{align}
for some constant~\(D_5 > 0,\) where the second relation in~(\ref{t_two_f}) is true since the edge weight~\(W(f)\) is independent of~\(G.\)

Plugging~(\ref{gen_ta}) and~(\ref{t_two_f}) into~(\ref{i_loc_22_split}), we then get
\begin{align}
I_{loc,22} &\leq D_6F_c(L) \mathbb{E}m^2_G \sum_{f\in K_n} \ind(f \in G) \nonumber\\
&\;\;\;\;\;\;\;\;\;\;\;\;\;\;+\;\;D_6F_c^{1-2/s_1}(L)\mathbb{E}m_G\sum_{f \in K_n} \ind(f \in G) \nonumber\\
&= D_6F_c(L) \mathbb{E}m^3_G + D_6F_c^{1-2/s_1}(L) \mathbb{E}m^2_G \label{hristova}
\end{align}
for some constant~\(D_6 > 0.\) To estimate moments of~\(m_G,\) the number of edges in~\(G,\) we recall the event~\(E_{nei}\) defined prior to~(\ref{e_nei_est}) which ensures that each vertex has at most~\(D_7 nr_n^2p_n\) neighbours in~\(G,\) for some constant~\(D_7 > 0.\) This ensures that the total number of edges in~\(G\) is at most~\(D_7n^2r_n^2p_n.\) If~\(E_{nei}^c\) occurs, then we have~\(m_G \leq n^2,\) an upper bound on the total number of edges in the complete graph~\(K_n.\)  Using the bound~(\ref{e_nei_est}) for~\(E_{nei},\) we then get
\[\mathbb{E}m_G^3 \leq D_8 (n^2r_n^2p_n)^3 + n^6e^{-D_9nr_n^2p_n}\]
and
\[\mathbb{E}m_G^2 \leq D_8 (n^2r_n^2p_n)^2 + n^4e^{-D_9nr_n^2p_n}\]
for some constants~\(D_8,D_9 > 0.\)

Substituting the above into~(\ref{hristova}) and using the respective estimates~(\ref{i_loc_est_21}) and~(\ref{i_loc_one_est}) for~\(I_{loc,21}\) and~\(I_{loc,1},\) we get
\begin{align} \label{i_loc_est_tot}
I_{loc} &= I_{loc,1} + I_{loc,2} \nonumber\\
&\leq I_{loc,1}+I_{loc,21} + I_{loc,22} \nonumber\\
&\leq D (nr_n^2p_n)^2 \left(L^2 + n^{4}r_n^2p_n F_c(L) + n^2F_c^{1-2/s_1}(L)\right) \nonumber\\
&\;\;\;\;\;\;\;\;\;+\;\;\;Dn^{6}\exp\left(-D_0nr_n^2p_n\right) \nonumber\\
&= D\sigma^2_{loc}
\end{align}
for some constants~\(D,D_0 > 0,\) where~\(\sigma^2_{loc}\) is as in the Theorem statement and the second relation in~(\ref{i_loc_est_tot}) follows from~(\ref{i_loc_22_split}). Thus~\(nI_{loc} \leq Dn\sigma^2_{loc}\) and this obtains the variance contribution due to the randomness in vertex locations (see~\ref{i_loc_est_one}).

\emph{\underline{Step 2}}: It remains to estimate the variance contribution~\[I_{wt} = \mathbb{E}\left(\tau_n-\tau_n^{(1)}\right)^2\] due to edge states and weights in~(\ref{i_loc_est_one}). For convenience, we recall that~\(f_1,\ldots,f_m, m = {n \choose 2}\) is a deterministic ordering of edges of~\(K_n\) and that~\(\tau_n^{(1)}\) is the maximum weight of a spanning tree of the ``new" random graph~\(G^{(1)}\) obtained by replacing the edge state and weight~\((Z(f_1),W(f_1))\) of the edge~\(f_1\) in~\(G,\) by an independent copy~\((Z^{(c)}(f_1),W^{(c)}(f_1)).\)

Let~\(f_1 = (u_1,v_1)\) have~\(u_1\) and~\(v_1\) as endvertices and let~\(G_{rem}(f_1)\) be the subgraph of~\(G\) obtained by removing only the edge~\(f_1\) from~\(G.\) Letting~\(\tau_{rem}(f_1)\) be the maximum weight of a spanning tree of the largest component of~\(G_{rem}(f_1),\) we get by an application of the triangle inequality that
\[|\tau_n - \tau_n^{(1)}| \leq |\tau_n - \tau_{rem}(f_1)| + |\tau_{rem}(f_1)-\tau_n^{(1)}|\] and so squaring, taking expectations and using~\((a+b)^2 \leq 2(a^2+b^2),\) we get that
\begin{align}
I_{wt} &\leq 2\left(\mathbb{E}\left(\tau_n - \tau_{rem}(f_1)\right)^2 + \mathbb{E}\left(\tau_{rem}(f_1) - \tau_n^{(1)}\right)^2\right) \nonumber\\
&= 4\mathbb{E}\left(\tau_n-\tau_{rem}(f_1)\right)^2. \label{i_wt_appax}
\end{align}

We estimate the difference~\(\tau_n-\tau_{rem}(f_1)\) as follows. From the proof of Step~\(1\) above, we recall that~\(G(\{u_1,v_1\})\) is the graph obtained by removing the \emph{vertices}~\(u_1\) and~\(v_1.\) Recalling that~\(E_{con}(\{u_1,v_1\})\) is the event that~\(G(\{u_1,v_1\})\) is connected, we get from~(\ref{e_con_est}) that
\begin{equation}\label{madai_ax}
\mathbb{P}(E_{con}(\{u_1,v_1\})) \geq 1-\exp\left(-Dnr_n^2p_n\right)
\end{equation}
for some constant~\(D > 0,\) not depending on the choice of~\(\{u_1,v_1\}.\) We also recall the event~\(E_{nei}\) defined prior to~(\ref{e_nei_est}) that ensures that the degree of each vertex in~\(G\) is at least of the order of~\(nr_n^2p_n.\) Defining
\[E_{nice} := E_{con}\left(\emptyset\right) \bigcap E_{con}(\{u_1,v_1\}) \bigcap E_{nei}\] and choosing the constant~\(D > 0\) in~(\ref{madai_ax}) smaller if necessary, we invoke the union bound and get from the corresponding probability estimates~(\ref{e_con_est}),~(\ref{madai_ax}) and~(\ref{e_nei_est}) that
\begin{equation}\label{e_nice_est}
\mathbb{P}(E_{nice}) \geq 1-3\exp\left(-Dnr_n^2p_n\right).
\end{equation}

We now split the upper bound~(\ref{i_wt_appax}) for~\(I_{wt}\) as
\begin{equation}\label{i_wt_split_ax}
I_{wt} \leq 4I_{wt,1} + 4I_{wt,2},
\end{equation}
where
\[I_{wt,1} := \mathbb{E}\left(\tau_n-\tau_{rem}(f_1)\right)^2\ind(E_{nice})\] and
\[I_{wt,2} := \mathbb{E}\left(\tau_n-\tau_{rem}(f_1)\right)^2\ind(E^c_{nice}).\] In what follows, we estimate~\(I_{wt,2}\) and~\(I_{wt,1}\) in that order below.

If~\(E_{nice}^c\) occurs, then we use the direct bound~\(\tau_n \leq \sum_{f \in K_n} W(f),\) the sum of weights of all edges in the complete graph~\(K_n.\) The same estimate holds for~\(\tau_{rem}(f_1)\) as well and so we get
\[|\tau_n - \tau_{rem}(f_1)| \leq \sum_{f \in K_n} W(f).\] Consequently
\begin{align}
I_{wt,2} &= \mathbb{E}\left(\tau_n - \tau_{rem}(f_1)\right)^2 \ind(E_{nice}^c) \nonumber\\
&\leq \mathbb{E}\left(\sum_{f \in K_n} W(f)\right)^2 \ind(E_{nice}^c) \nonumber\\
&= \mathbb{E}\left(\sum_{f \in K_n}W(f)\right)^2 \mathbb{P}\left(E_{nice}^c\right), \label{sakiye}
\end{align}
since~\(E_{nice}\) depends only on the vertex locations and edge states and is therefore independent of edge weights.

Using~\(\left(\sum_{i=1}^{l}a_i\right)^2 \leq l \sum_{i=1}^{l}a_i^2\) and recalling that there are~\(m={n \choose 2}\) edges in~\(K_n,\) we get that
\begin{align}
\mathbb{E}\left(\sum_{f \in K_n} W(f)\right)^2 &\leq m\sum_{f \in K_n} \mathbb{E}W^2(f)  \nonumber\\
&\leq D_1m^2 \nonumber\\
&= D_1n^4, \nonumber
\end{align}
for some constant~\(D_1> 0,\) since the edge weights have bounded second moments by Theorem statement. Plugging this into~(\ref{sakiye}) and using the estimate~(\ref{e_nice_est}) for~\(E_{nice}\) we get that
\begin{equation}\label{i_wt_two_est}
I_{wt,2} \leq D_1 n^4 e^{-Dnr_n^2p_n} \leq e^{-D_2 nr_n^2p_n}
\end{equation}
for some constant~\(D_2 > 0,\) provided~\(nr_n^2p_n \geq M\log{n}\) for large enough constant~\(M.\) This obtains an upper bound for~\(I_{wt,2}.\)

To estimate~\(I_{wt,1},\) we assume henceforth that~\(E_{nice}\) occurs so that both~\(G\) and~\(G(\{u_1,v_1\})\) are connected. Because~\(E_{nei}\) also occurs both~\(u_1\) and~\(v_1\) are adjacent to at least~\(D_0 nr_n^2p_n\) vertices in~\(G.\) Therefore the connectivity of~\(G(\{u_1,v_1\})\) ensures that~\(G_{rem}(f_1)\) is connected as well and we let~\({\cal T}_n\) and~\({\cal T}_{rem}(f_1)\) be the maximum weight spanning trees of~\(G\) and~\(G_{rem}(f_1),\) respectively. Clearly, any spanning tree of~\(G_{rem}(f_1)\) is also a spanning tree of~\(G\) and so~\({\cal T}_{rem}(f_1)\) has weight at most~\(\tau_n;\) i.e.,
\begin{equation}\label{tim_one}
\tau_{rem}(f_1) \leq \tau_n.
\end{equation}

For the reverse direction, we see that~\(\tau_{rem}(f_1) < \tau_n\) only if~\(f_1 \in {\cal T}_n.\) If we remove~\(f_1\) from~\({\cal T}_n,\) then we get two subtrees~\({\cal R}_a\) and~\({\cal R}_b\) of~\({\cal T}_n,\) that are also trees in~\(G_{rem}(f_1).\) Since~\(G_{rem}(f_1)\) is connected, there must exist an edge~\(h_{ab} \in G_{rem}(f_1)\) such that the union~\[{\cal T}_{ab} := {\cal R}_a \cup \{h_{ab}\} \cup {\cal R}_b\] is connected (see Figure~\ref{fig_sub_trees}) and accompanying discussion. The tree~\({\cal T}_{ab}\) is a spanning tree of~\(G_{rem}(f_1)\) and has weight at least~\(\tau_n - W(f_1)\) and so we get that
\begin{equation}\label{tim_two}
\tau_{rem}(f_1) \geq \tau_n - W(f_1).
\end{equation}

Combining~(\ref{tim_two}) with~(\ref{tim_one}), we get that
\[|\tau_n-\tau_{rem}(f_1)| \ind(E_{nice}) \leq W(f_1) \ind(f_1 \in {\cal T}_n)\] and so squaring and taking expectations, we get that
\begin{equation}\label{anegan_ax}
I_{wt,1} \leq \mathbb{E}W^2(f_1)\ind\left(f_1 \in {\cal T}_n\right)= \frac{\mathbb{E}\rho_n}{m},
\end{equation}
where~\[\rho_n := \sum_{f \in K_n} W^2(f)\ind(f \in {\cal T}_n)\] is the sum of \emph{squares} of edge weights in the maximum weight spanning tree~\({\cal T}_n\) and the final estimate in~(\ref{anegan_ax}) follows from symmetry. As defined before,~\(m = {n \choose 2}\) is the number of edges in~\(K_n.\)

Arguing as in~(\ref{tau_later}), we have that
\[\rho_n \leq 4ny_n^2 + 16y_n^2\sum_{j \geq 1}j^2N_{bad}(j)\] where we recall that~\(y_n = H(nr_n^2p_n)\) and~\(N_{bad}(j)\) is the number of~\(j-\)bad edges in~\(G;\) i.e., the number of edges whose weight lies in~\([2jy_n,2(j+1)y_n).\) By an analogous analysis described between~(\ref{e_tau_n_up}) and~(\ref{e_tau_n_ax}) and using the fact that the scaling relation~(\ref{dilpax}) holds for some~\(s > 3\) strictly, we then get~\[\mathbb{E}\rho_n \leq D n H^2(nr_n^2p_n)\] for some constant~\(D  >0,\) provided~\(nr_n^2p_n \geq M\log{n}\) for some large enough constant~\(M  > 0.\) Plugging this into~(\ref{anegan_ax}), we get
\begin{equation}\label{i_one_est_ax}
mI_{wt,1} \leq 4DnH^2(nr_n^2p_n)
\end{equation}

Combining~(\ref{i_one_est_ax}) with the estimate~(\ref{i_wt_two_est}) for~\(I_{wt,2}\) we get from~(\ref{i_wt_split_ax}) that
\[mI_{wt} \leq 4mI_{wt,1} + 4mI_{wt,2} \leq 16DnH^2(nr_n^2p_n) + 4me^{-D_2nr_n^2p_n},\] where we recall that~\(m=n^2.\) Arguing as in the discussion following~(\ref{jelly}) in Theorem~\ref{thm_mst_max}, we see that~\(H(nr_n^2p_n)\) is uniformly bounded away from zero and so if~\(nr_n^2p_n^2 \geq M\log{n}\) for a large enough constant~\(M,\) then~\[mI_{wt} \leq 2DnH^2(nr_n^2p_n) = 2Dn\sigma_{wt}^2\] for all~\(n\) large. From~(\ref{i_loc_est_one}), we see that this obtains the desired variance contribution estimate due to randomness in edge states and weights and therefore completes the proof of the Theorem.~\(\qed\)


\subsection*{\em Acknowledgement}
I thank Professors Rahul Roy, Federico Camia, Alberto Gandolfi and C. R. Subramanian for crucial comments and also thank IMSc and IISER Bhopal for my fellowships.

\subsection*{\em Data Availability Statement}
Data sharing not applicable to this article as no datasets were generated or analysed during the current study.

\subsection*{\em Conflict of Interest}
The authors have no conflicts of interest to declare that are relevant to the content of this article. No funding was received to assist with the preparation of this manuscript.

\bibliographystyle{plain}

\begin{thebibliography}{10}
\bibitem{berry} L. Addario-Berry, N. Broutin, C. Goldschmidt, and G. Miermont. (2017).
\newblock{The Scaling Limit of the Minimum Spanning Tree of the Complete Graph}.
\newblock{\em Annals of  Probability}, \textbf{45}, 3075--3144.

\bibitem{aldous} D. Aldous. (1990).
\newblock{ A Random Tree Model Associated with Random Graphs}.
\newblock{\em Random Structures and Algorithms}, \textbf{1}, 383--202.


\bibitem{alon} N. Alon and J. Spencer. (2008).
\newblock{\em The Probabilistic Method}.
\newblock{Wiley}.

\bibitem{boll} B. Bollob\'as. (2001).
\newblock{\em Random Graphs}.
\newblock{Cambridge University Press}.

\bibitem{frieze} A. M. Frieze. (1985).
\newblock{On the Value of a Random Minimum Spanning Tree Problem}.
\newblock{\em Discrete Applied Mathematics}, \textbf{10}, 47--56.


\bibitem{ganesan} G. Ganesan. (2020).
\newblock{Minimum Spanning Trees of Random Geometric Graphs With Location Dependent Edge Weights}.
\newblock{\em Bernoulli}, \textbf{27}, 2473--2493.

\bibitem{goldsmith} A. Goldsmith. (2005).
\newblock{\em Wireless Communications}.
\newblock{Cambridge University Press}.

\bibitem{kes_lee} H. Kesten  and S. Lee. (1996).
\newblock{The Central Limit Theorem for Weighted Minimal Spanning Trees on Random Points}.
\newblock{\em Annals of Applied Probability}, \textbf{6}, 495--527.


\bibitem{penrose} M. Penrose. (2003).
\newblock{\em Random Geometric Graphs}.
\newblock{Oxford University Press}.

\bibitem{penrose2} M. Penrose and J. Yukich. (2003).
\newblock{Weak Laws of Large Numbers in Geometric Probability}.
\newblock{\em  Annals of Applied Probability}, \textbf{13}, 277--303.

\bibitem{steele} J. Steele. (1993).
\newblock{Probability and Problems in Euclidean Combinatorial Optimization}.
\newblock{\em Statistical Science}, \textbf{8}, 48--56.

\bibitem{steele2} J. Steele. (2002).
\newblock{Minimal Spanning Trees for Graphs with Random Edge Lengths}.
\newblock{\em Mathematics and Computer Science II. Algorithms, Trees, Combinatorics and Probabilities}, pp. 223–-245.

\bibitem{yukich} J. E. Yukich. (1998).
\newblock{Probability Theory of Classical Euclidean Optimization Problems}.
\newblock{\em Lecture Notes in Mathematics, Springer}, \textbf{1675}.




\end{thebibliography}

\end{document}